\documentclass[12pt,letterpaper]{amsart}

\newcommand{\indentalign}{\hspace{0.3in}&\hspace{-0.3in}}
\newcommand{\la}{\langle}
\newcommand{\ra}{\rangle}
\renewcommand{\Re}{\operatorname{Re}}
\renewcommand{\Im}{\operatorname{Im}}

\newcommand{\sech}{\operatorname{sech}}
\newcommand{\defeq}{\stackrel{\rm{def}}{=}}

\newcommand{\spn}{\operatorname{span}}

\usepackage{amssymb}
\usepackage{amsthm}
\usepackage{amsxtra}
\usepackage{graphicx}

\newtheorem{theorem}{Theorem}

\newtheorem{lemma}[theorem]{Lemma}

\theoremstyle{remark}

\numberwithin{equation}{section}

\numberwithin{theorem}{section}

\numberwithin{table}{section}

\numberwithin{figure}{section}

\ifx\pdfoutput\undefined
  \DeclareGraphicsExtensions{.pstex, .eps}
\else
  \ifx\pdfoutput\relax
    \DeclareGraphicsExtensions{.pstex, .eps}
  \else
    \ifnum\pdfoutput>0
      \DeclareGraphicsExtensions{.pdf}
    \else
      \DeclareGraphicsExtensions{.pstex, .eps}
    \fi
  \fi
\fi

\title[NLS soliton interaction]
{Phase-driven interaction of widely separated nonlinear Schr\"odinger solitons}

\author{Justin Holmer}
\author{Quanhui Lin}
\address{Brown University}

\begin{document}

\maketitle

\begin{abstract}
We show that, for the 1d cubic NLS equation, widely separated equal amplitude in-phase solitons attract and opposite-phase solitons repel.  Our result gives an exact description of the evolution of the two solitons valid until the solitons have moved a distance comparable to the logarithm of the initial separation.  Our method does not use the inverse scattering theory and should be applicable to nonintegrable equations with local nonlinearities that support solitons with exponentially decaying tails.  The result is presented as a special case of a general framework which also addresses, for example, the dynamics of single solitons subject to external forces as in \cite{HZ1, HZ2}.
\end{abstract}

\section{Introduction}
\label{S:introduction}

We consider the 1d nonlinear Schr\"odinger equation (NLS)
\begin{equation}
\label{E:NLS}
i\partial_t u + \tfrac12 \partial_x^2 u + |u|^2u=0 \,.
\end{equation}
It has a single soliton solution $u(x,t)=e^{it/2}\sech x$.  The invariances of \eqref{E:NLS} can be applied to produce a whole family of solutions.  To describe them, let
\footnote{We order the parameters as $(\mu,a,\theta,v)$ to mimic $(q^1,q^2,p^1,p^2)$ as canonical coordinates for the four dimensional symplectic space with symplectic form $dp^1\wedge dq^1 + dp^2 \wedge dq^2$.}
\begin{equation}
\label{E:eta}
\eta(x,\mu,a,\theta,v) = e^{i\theta}e^{i\mu^{-1}v(x-a)}\mu \sech(\mu (x-a))
\end{equation}
for parameters $\theta, a, v \in \mathbb{R}$, $\mu>0$.  Then $u(x,t) = \eta(x,\mu(t),a(t),\theta(t),v(t))$ solves \eqref{E:NLS} provided
\begin{equation}
\label{E:free-flow}
\left\{ 
\begin{aligned}
&\mu(t) = \mu_0 \\
&a(t) = a_0 + tv_0\mu_0^{-1} \\
&\theta(t) = \theta_0 + \frac12t(\mu_0^2+\mu_0^{-2}v_0^2) \\
&v(t) = v_0
\end{aligned} 
\right.
\end{equation}

In this paper, we study the evolution of initial data that is the sum of two widely separated solitons:
\begin{equation}
\label{E:initial-data}
\begin{aligned}
u_0(x) &= \eta(x,\mu_{10},a_{10},\theta_{10},v_{10}) + \eta(x,\mu_{20},a_{20},\theta_{20},v_{20})
\end{aligned}
\end{equation}
where $|a_{20}-a_{10}| \gg 1$.   In particular, we focus on two illustrative cases.  In both cases, we consider identical mass solitons with zero initial velocity. In Case 0, we take the same initial phase, corresponding to an even superposition and in Case 1, we take opposite initial phase corresponding to an odd superposition.
\begin{equation}
\label{E:case01}
u_0(x) = 
\left\{
\begin{aligned}
&\eta(x,1,-a_0,0,0) + \eta(x,1,a_0,0,0) && \text{Case }\sigma=0\\
&\eta(x,1,-a_0,\pi,0) + \eta(x,1,a_0,0,0) && \text{Case }\sigma=1
\end{aligned}
\right.
\end{equation}

\newcommand{\sol}{\textnormal{sol}}
\newcommand{\eff}{\textnormal{eff}}

We find that in the same phase case (Case 0), the two solitons are drawn toward each other and in the opposite phase case (Case 1) they are pushed apart-- see Fig. \ref{F:1}.  In either case, the solution $u$ to \eqref{E:NLS} is well-approximated by
\begin{equation}
\label{E:V-10}
u_z(x) = \eta(x,\mu_1,a_1,\theta_1,v_1) + \eta(x,\mu_2,a_2,\theta_2,v_2)
\end{equation}
where $z$ represents coordinates\footnote{Superscripts are used on $z$ to conform with geometric summation conventions used later in the paper.}
\begin{equation}
\label{E:V-11}
z = (z^1,z^2,z^3,z^4,z^5,z^6,z^7,z^8) = (\mu_1,a_1,\mu_2,a_2,\theta_1,v_1,\theta_2,v_2)
\end{equation}
As parity is preserved by the flow \eqref{E:NLS}, we have
\begin{equation}
\label{E:V-12}
\mu\defeq \mu_1=\mu_2 \,, \quad a\defeq -a_1=a_2 \,, \quad v \defeq -v_1=v_2 \,,
\end{equation}
and $\theta \defeq \theta_1 =\theta_2$ in the same phase case (Case 0), while $\theta \defeq \theta_1-\pi=\theta_2$ in the opposite phase case (Case 1).

\begin{figure}
\begin{center}
\includegraphics[scale=0.4]{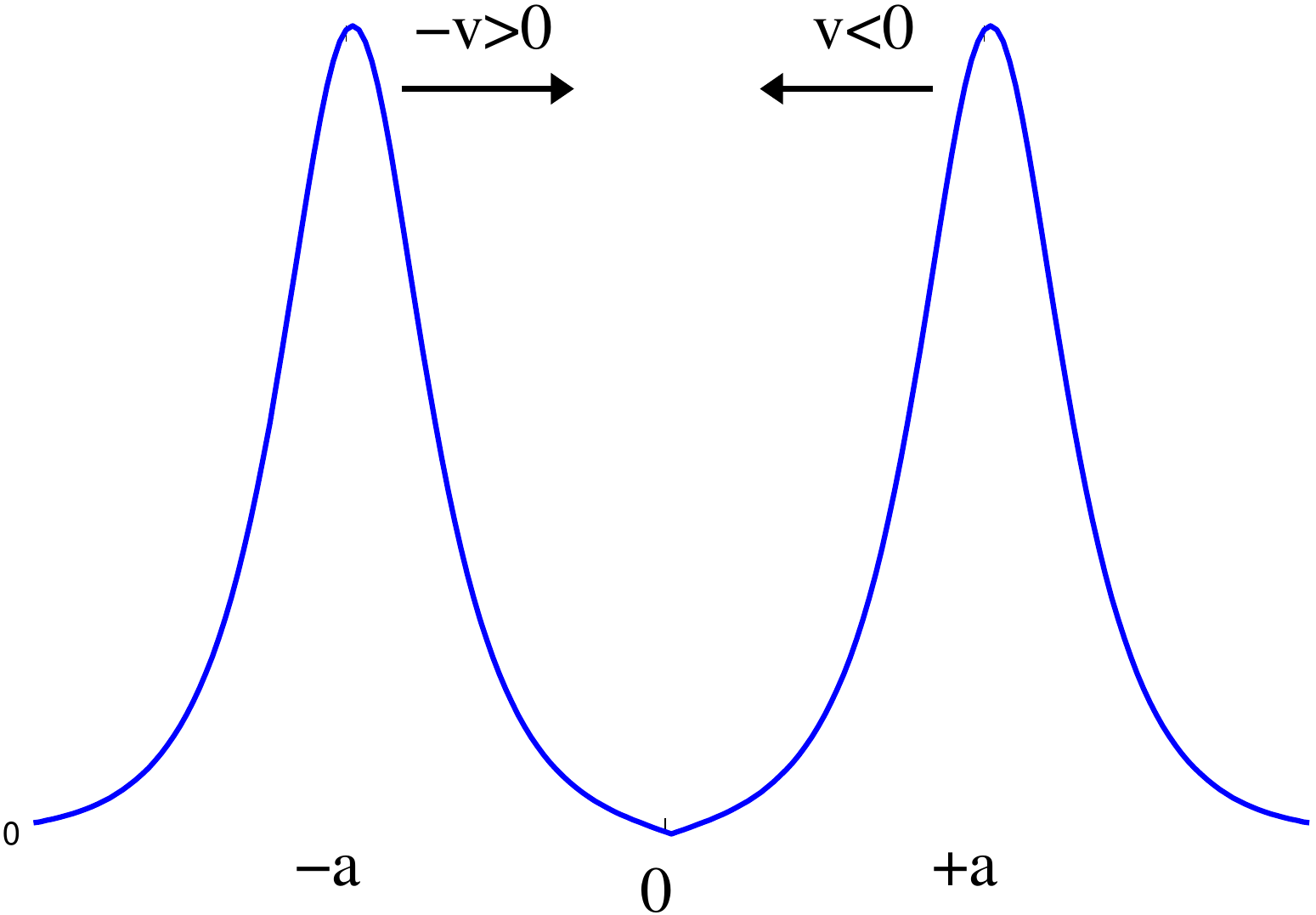}\\
\bigskip
\includegraphics[scale=0.4]{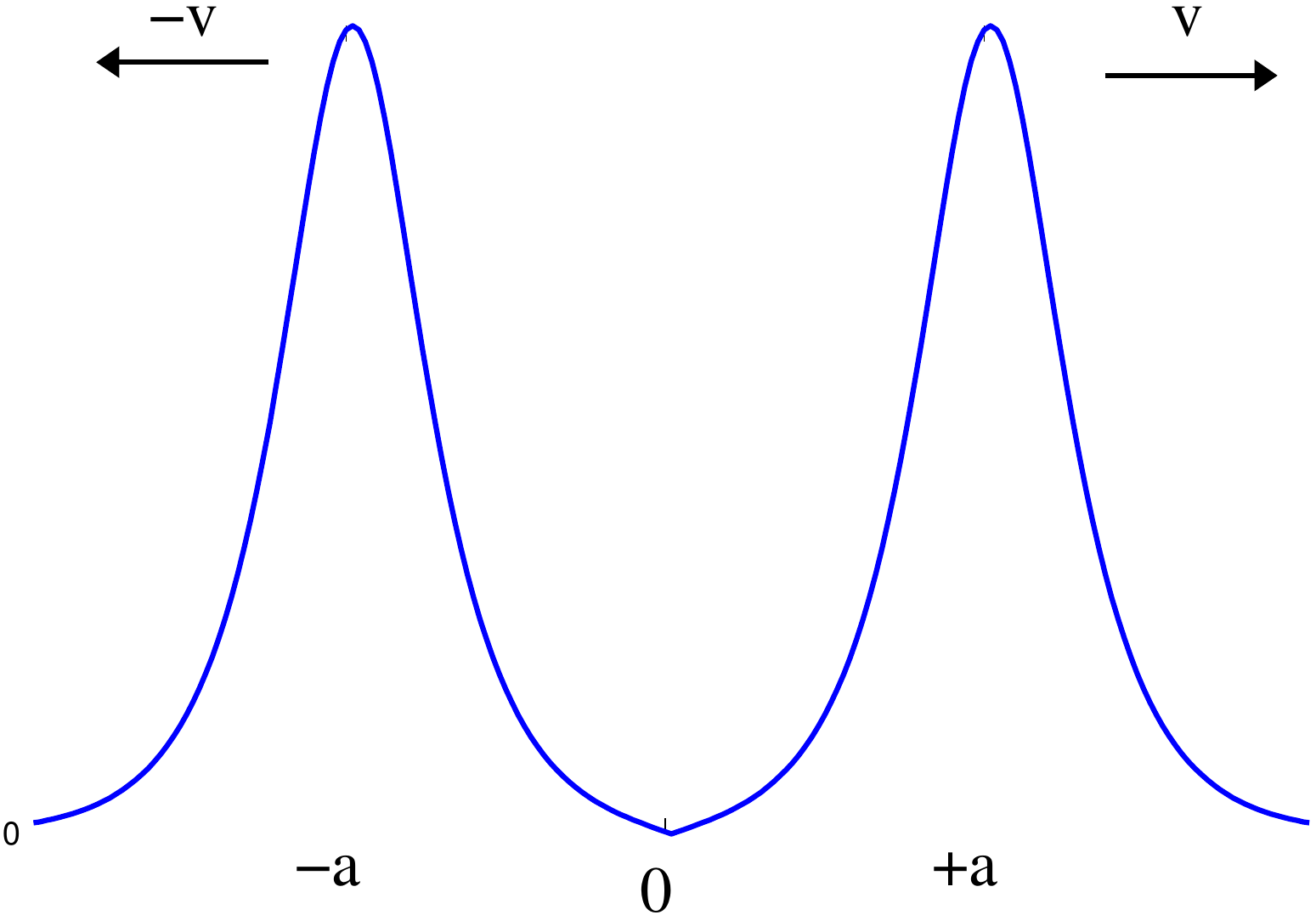}
\end{center}
\caption{
\label{F:1}
The top plot is a depiction of Case 0 (same phase; even solution), where the two solitons are pulled toward each other.  The bottom plot depicts Case 1 (opposite phase; odd solution), where they repel. In each case, the solution is modeled in Theorem \ref{T:main1} as $u\approx u_z = \eta(\mu,-a,\theta+\sigma \pi,-v)+\eta(\mu,a,\theta,v)$ where $(\mu,a,\theta,v)$ solve a specific ODE system.
}
\end{figure}

\begin{theorem}
\label{T:main1}
Suppose that $u(t)$ is the solution to \eqref{E:NLS} with initial data \eqref{E:case01}.  Let $h=e^{-a_0}\ll 1$ (so $a_0=\log h^{-1}\gg 1$).  Let 
$$
T \sim
\begin{cases}
h^{-1} & \text{Case }\sigma=0 \\
h^{-1}\log h^{-1} & \text{Case }\sigma=1 
\end{cases} 
$$
Let $(a(t),v(t))$ solve 
\begin{equation}
\label{E:aveqn}
\left\{
\begin{aligned}
&\dot a = v \\
&\dot v = -4(-1)^\sigma e^{-2a}
\end{aligned}
\right.
\end{equation}
with initial data $(a_0,0)$.  Let $\mu$ solve
\begin{equation}
\label{E:V-6}
\dot \mu = (-1)^\sigma (8a-4)ve^{-2a} \,,
\end{equation}
and then let $\theta$ solve
\begin{equation}
\label{E:V-7}
\dot \theta = \frac12 \mu^2 + \frac12v^2\mu^{-2} + 18(-1)^\sigma e^{-2a}\,.
\end{equation}
Then on $0\leq t \leq T$, we have
$$\|u(t)-u_z \|_{H_x^1} \lesssim h^{2-} \,,$$
where
\begin{equation}
\label{E:V-5}
u_z = \eta(\mu,-a,\theta+\sigma \pi,-v)+\eta(\mu,a,\theta,v)
\end{equation}
\end{theorem}

Let us make some remarks on the ODE system \eqref{E:aveqn}.  The energy associated to this system is
$$H_\eff = v^2 - 4(-1)^\sigma e^{-2a} = -4(-1)^\sigma e^{-2a_0}$$
In the case $\sigma =0$ (same phase), we have $v\leq 0$, $a\leq a_0$ and 
$$
\left\{
\begin{aligned}
&a(t)=a_0-\log \sec (2ht)\\
&v(t) = 2h\tan(2ht)
\end{aligned}
\right.
$$
valid for $0\leq t\lesssim e^{a_0}=h^{-1}$.
In the case $\sigma =1$ (opposite phase), we have $v\geq 0$, $a\geq a_0$ and
$$
\left\{
\begin{aligned}
&a(t)=a_0+\log \cosh(2ht) \\
&v(t)=2h\tanh(2ht)
\end{aligned}
\right.
$$
valid for $0\leq t\lesssim a_0e^{a_0} = h^{-1}\log h^{-1}$.  In either case, $\mu$ evolving according to \eqref{E:V-6} satisfies $|\mu-1| \lesssim h^2$ and can thus be replaced by $1$ in \eqref{E:V-5}.  However, the $O(h^2)$ behavior of $\mu$ is dynamically significant in that it yields $O(h^1)$ effects in $\theta$ through \eqref{E:V-7}.  It is evident from the explicit forms for $a(t)$ given above that, on the indicated time scale, the soliton has moved a distance comparable to $\log a_0$.

We remark that although \eqref{E:NLS} is completely integrable, we do not use the inverse scattering theory of Zakharov-Shabat \cite{ZS}.  We expect that one could compute the scattering data associated to our initial condition and conduct an analysis using inverse scattering theory that would describe the dynamics for all time.  Our argument, however, has the merit of being relatively simple and should adapt to most nonintegrable nonlinearities that support stable solitons with exponentially decaying tails.  An important example of such a nonintegrable equation is the 1d cubic-quintic NLS:
$$i\partial_t u + \frac12 \partial_x^2 u + |u|^2u - \epsilon |u|^4u=0$$
Furthermore, our goal was not just to obtain Theorem \ref{T:main1} but to present it in the conceptual (yet rigorous) framework of symplectic restriction that illustrates its connection to previous work of the first author, Holmer-Zworski \cite{HZ1, HZ2}.

We cite two papers from the physics literature as motivation for our problem.  Stegeman-Segev \cite{SS} provide an overview of phase-driven two-soliton interaction in the context of optics, beginning with an account of the 1d case \eqref{E:NLS} that we study (see their Fig. 4) and proceeding to a discussion of two-soliton interaction in two dimensions in which the attractive forces between in-phase solitons can lead to spiraling structures -- see their Fig. 6.  The NLS equation  also arises in a completely different physical setting, Bose-Einstein condensation.  Strecker et.al. \cite{Str} describe an experiment producing multiple solitons, in which the model is \eqref{E:NLS} with a confining potential.  A train of five solitons with successively opposite phases are produced and oscillate in a well.  At the peak of the oscillations, the solitons bunch up but retain some separation; \cite{Str} explains this in terms of their phase differences.

We will now give an explanation of Theorem \ref{T:main1} and an overview of the proof.  Consider $L^2(\mathbb{R};\mathbb{C})$ as a manifold with metric 
$$g_u(v_1,v_2) = \la v_1,v_2\ra \defeq \Re \int v_1 \bar v_2 \quad \text{for} \quad u\in L^2, \; v_1,v_2 \in T_uL^2 \simeq L^2 \,.$$
Introduce $J=-i$, viewed as an operator  $T_uL^2\to T_uL^2$.  The corresponding symplectic form is
\begin{equation}
\label{E:V-16}
\omega_u(v_1,v_2) = \la v_1, J^{-1}v_2\ra
\end{equation}
Take as Hamiltonian the (densely defined, with domain\footnote{This domain is chosen so that $JH'(u) = -i(-\frac12 u_{xx}-|u|^2u) \in L^2$.  Although we restrict to $u\in D=H^2$ here, we will prove estimates on the corresponding flow in $H^1$.  This parallels the situation in the theory of linear self-adjoint operators $A$, where a dense domain is specified but the flow associated to $-iA$ extends to a unitary operator on all of $L^2$.}
 $D=H^2$) function $H:L^2\to \mathbb{R}$ given by
\begin{equation}
\label{E:V-15}
H(u) = \frac14 \int |u_x|^2 - \frac14 \int |u|^4
\end{equation}
Then
 $$H'(u) \in T_u^*L^2  \underset{\text{metric}\; g}\simeq T_uL^2$$
The corresponding flow is $\partial_t u = JH'(u)$ yielding \eqref{E:NLS}.

Recalling that $\eta$ is given by \eqref{E:eta}, consider the manifold of solitons
$$ M = \{ \, \eta(\cdot, \mu,a,\theta,v) \, | \, \,\,\mu>0, \theta\in \mathbb{R}, a\in \mathbb{R}, v\in \mathbb{R} \, \} \,.$$
Computations show that the restriction of the symplectic form $\omega$ to $M$ is
$$i^*\omega = d\theta\wedge d\mu + dv\wedge da \,,$$
while the restriction of the Hamiltonian $H$ to $M$ is
$$H(\eta(\cdot, \mu,a,\theta,v)) = \frac12 \mu^{-1}v^2 - \frac16\mu^3 \,,$$
Note that the free single soliton flow \eqref{E:free-flow} is just the solution to the Hamilton equations of motion for $H(\eta)$ with respect to $i^*\omega$:
$$
\left\{
\begin{aligned}
&\dot \mu = \partial_\theta H(\eta) =0 \\
&\dot a = \partial_v H(\eta) = \mu^{-1}v\\
& \dot \theta = - \partial_\mu H(\eta) = \frac12 \mu^{-2}v^2 + \frac12 \mu^2 \\
&\dot v = -\partial_a H(\eta) =0
\end{aligned}
\right.
$$

Turning to the double soliton problem, recall that we model the $u$ in terms of $u_z$ given by \eqref{E:V-10} where $z=(z^1, \ldots, z^8)$ is given by \eqref{E:V-11}.  We introduce the shorthand notation
$$\eta_j \defeq \eta(\cdot, \mu_j, a_j, \theta_j, v_j) \,, \quad j=1,2.$$
Also recall that $h=e^{-a_0} \ll 1$, and the initial soliton separation is $2a_0= 2\log h^{-1} \gg 1$.  Expanding the nonlinearity, we obtain
\begin{equation}
\label{E:237}
H_p(u_z) = H_p(\eta_1) +\la H_p'(\eta_1),\eta_2 \ra + H_p(\eta_2) + \la H_p'(\eta_2),\eta_1\ra + O(h^4)
\end{equation}
The last two terms are dominant near $a_2$ (on the effective support of $\eta_2$), so that the second soliton sees an ``effective'' Hamiltonian 
\begin{equation}
\label{E:239}
H_{\text{eff}}(\mu_2,a_2,\theta_2,v_2) = H(\eta_2) + \la H_p'(\eta_2),\eta_1\ra
\end{equation}
and thus its expected equations of motion are
\begin{equation}
\label{E:V-1}
\left\{
\begin{aligned}
&\dot \mu_2 = \partial_{\theta_2}H(\eta_2) + \partial_{\theta_2}\la H_p'(\eta_2),\eta_1\ra \\
&\dot a_2 = \partial_{v_2} H(\eta_2) + \partial_{v_2} \la H_p'(\eta_2),\eta_1\ra\\
&\dot \theta_2 = -\partial_{\mu_2} H(\eta_2) - \partial_{\mu_2} \la H_p'(\eta_2),\eta_1\ra\\
&\dot v_2 = -\partial_{a_2} H(\eta_2) - \partial_{a_2}\la H_p'(\eta_2),\eta_1\ra
\end{aligned}
\right.
\end{equation}
Likewise, the first two terms in \eqref{E:237} are dominant near $a_1$ so the first soliton sees an effective Hamiltonian 
$$H_{\text{eff}}(\mu_1,a_1,\theta_1,v_1) = H(\eta_1) + \la H_p'(\eta_1),\eta_2\ra$$
and thus its expected equations of motion are
\begin{equation}
\label{E:V-2}
\left\{
\begin{aligned}
&\dot \mu_1 = \partial_{\theta_1}H(\eta_1) + \partial_{\theta_1}\la H_p'(\eta_1),\eta_2\ra \\
&\dot a_1 = \partial_{v_1} H(\eta_1) + \partial_{v_1} \la H_p'(\eta_1),\eta_2\ra\\
&\dot \theta_1 = -\partial_{\mu_1} H(\eta_1) - \partial_{\mu_1} \la H_p'(\eta_1),\eta_2\ra\\
&\dot v_1 = -\partial_{a_1} H(\eta_1) - \partial_{a_1}\la H_p'(\eta_1),\eta_2\ra
\end{aligned}
\right.
\end{equation}
Pulling \eqref{E:V-1} and \eqref{E:V-2} together gives us a systems of eight equations in eight unknowns:
\begin{equation}
\label{E:V-3}
\left\{
\begin{aligned}
&\dot \mu_1 = \partial_{\theta_1}H(\eta_1) + \partial_{\theta_1}\la H_p'(\eta_1),\eta_2\ra \\
&\dot a_1 = \partial_{v_1} H(\eta_1) + \partial_{v_1} \la H_p'(\eta_1),\eta_2\ra\\
&\dot \theta_1 = -\partial_{\mu_1} H(\eta_1) - \partial_{\mu_1} \la H_p'(\eta_1),\eta_2\ra\\
&\dot v_1 = -\partial_{a_1} H(\eta_1) - \partial_{a_1}\la H_p'(\eta_1),\eta_2\ra\\
&\dot \mu_2 = \partial_{\theta_2}H(\eta_2) + \partial_{\theta_2}\la H_p'(\eta_2),\eta_1\ra \\
&\dot a_2 = \partial_{v_2} H(\eta_2) + \partial_{v_2} \la H_p'(\eta_2),\eta_1\ra\\
&\dot \theta_2 = -\partial_{\mu_2} H(\eta_2) - \partial_{\mu_2} \la H_p'(\eta_2),\eta_1\ra\\
&\dot v_2 = -\partial_{a_2} H(\eta_2) - \partial_{a_2}\la H_p'(\eta_2),\eta_1\ra
\end{aligned}
\right.
\end{equation}
 After the even/odd symmetry assumption is imposed, one can distill from \eqref{E:V-3} the equations appearing in the statement of Theorem \ref{T:main1}.
\footnote{In fact, the above heuristic argument does not invoke the even/odd symmetry assumption and thus we might expect the equations \eqref{E:V-3} even without this assumption.  However, the equations \eqref{E:V-3} are only expected to be accurate to order $O(h^4)$.  In the presence of the symmetry assumption the eight equations in \eqref{E:V-3} dramatically decouple as \eqref{E:aveqn}, \eqref{E:V-6}, \eqref{E:V-7} which permits a direct analysis of these ODEs that shows that an $O(h^4)$ unknown can only only have a limited $O(h^2)$ effect the solution.  In the general case, the eight equations in \eqref{E:V-3} are more interdependent and we are not certain as to the effect of $O(h^4)$ perturbations.  This is not the only obstacle to removing the symmetry assumption; see comments below.}

We find that the above argument yielding \eqref{E:V-3} is a little too vague to adapt to a rigorous proof.  We now consider a different perspective that informally produces the same set of equations \eqref{E:V-3} but adapts to yield a proof of Theorem \ref{T:main1} and in fact extends and unifies the results of \cite{FGJS, HPZ, HZ1, HZ2}.  Recalling $z$ defined in \eqref{E:V-11}, consider now the eight-dimensional two-soliton manifold 
$$M = \{ \, u_z=\eta_1+\eta_2 = \eta(\cdot, \mu_1,a_1,\theta_1,v_1)+\eta(\cdot, \mu_2,a_2,\theta_2,v_2) \,  \}$$
The symplectic form \eqref{E:V-16} restricted to $M$ is
\begin{equation}
\label{E:V-17}
i^*\omega = \frac12 \sum_{\ell,m=1}^8 a_{\ell m}(z) \, dz^\ell \wedge dz^m
\end{equation}
where 
$$A(z) = (a_{\ell m}(z)) \,, \qquad a_{\ell m}(z) = \la \partial_{z^\ell}u_z, J^{-1} \partial_{z^m} u_z \ra$$
Let $H(u_z)$ denote the restriction to $M$ of the Hamiltonian \eqref{E:V-15}.  The expected equations of motion for $z^m$ are Hamilton's equations for $H(u_z)$ with respect to $i^*\omega$.  These equations are:
\begin{equation}
\label{E:V-20}
\dot z^m = -\sum_{\ell=1}^8 \partial_{z^\ell} H(u_z) \, a^{\ell m}(z) \,, \quad m=1,\ldots, 8
\end{equation}
where $a^{\ell m}$ denotes the components of the inverse of the matrix $A=(a_{\ell m})$.  

The matrix $A$ contains $O(h^2)$ terms that result from the pairing of directions parallel to the first soliton with directions parallel to the second soliton.  Moreover, $H(u_z)$ contains additional $O(h^2)$ terms arising from the quadratic part of \eqref{E:V-15} not represented in \eqref{E:V-3}.  It turns out that $O(h^2)$ terms in $a^{\ell m}$  and $O(h^2)$ terms in $H(u_z)$ each give rise to terms which cancel in \eqref{E:V-20}.  This hinges upon the fact that
\begin{equation}
\label{E:422}
\partial_{z^\ell} H(u_z) = -\sum_{j=1}^8 b^j a_{j\ell} + \partial_{z^\ell}\la \eta_1, H_p'(\eta_2)\ra + \partial_{z^\ell} \la \eta_2, H_p'(\eta_1)\ra 
\end{equation}
where
$$
b_2 = \partial_{v_1}H(\eta_1)\,, \quad 
b_4 = \partial_{v_2}H(\eta_2) \,, \quad 
b_5 = -\partial_{\mu_1}H(\eta_1) \,, \quad 
b_7 = -\partial_{\mu_2}H(\eta_2)$$
and all other $b_j=0$.  When this equation is substituted into \eqref{E:V-20}, once can witness the simplification arising from the pairing of $A$ and $A^{-1}$, and this shows that \eqref{E:V-20} is equivalent to \eqref{E:V-3}.  We elaborate upon this in Appendix A.

The merit in this point of view is that the equations \eqref{E:V-20} readily follow from the symplectic decomposition of the flow--that is, we select $z$ (via the implicit function theorem) so that 
\begin{equation}
\label{E:238}
u = u_z+w
\end{equation}
where $w\in T_zM^\perp$ (the \emph{symplectic} orthogonal complement to $T_zM$ in $T_{u_z}L^2$).  In \S \ref{S:eff-dyn} (Lemma \ref{L:ODEs}) a general argument is given showing that the equations \eqref{E:V-20} follow, with errors of size $h^4+\|w\|_{H^1}^2$.  This argument exploits the fact that \eqref{E:V-20}, with errors of size $h^4+\|w\|_{H^1}^2$, is equivalent to 
\begin{equation}
\label{E:intro10}
\partial_t u_z = \Pi_z JH'(u_z) + O(h^4+\|w\|_{H^1}^2)
\end{equation}
where $\Pi_z:T_{u_z}L^2 \to T_zM$ is the symplectic orthogonal projection operator given explicitly by
$$
\Pi_z f = \sum_{\ell, m=1}^8  \la f, J^{-1}\partial_{z^\ell} u_z \ra a^{\ell m}(z) \partial_{z^m}u_z
$$
The proof of Lemma \ref{L:ODEs} makes no reference to the specific meaning of $H$ or $u_z$, and a similar result with nearly identical proof would yield the equations of motion in many other problems, including those studied in \cite{HPZ, HZ1, HZ2}.  The fact that the equations of motion follow automatically but rigorously from the symplectic decomposition \eqref{E:238} is one of the main advantages of this geometric approach to our problem, as opposed to a more \emph{ad hoc} approach based on the discussion surrounding \eqref{E:239}.  \footnote{The idea that the equations of motions should be Hamilton's equations for the restricted Hamiltonian with respect to the restricted symplectic form was introduced in \cite{HZ1, HZ2} and supported informally with an argument involving Darboux's theorem.  The equations of motion thus obtained were used as a guide in the analysis in \cite{HPZ, HZ1, HZ2} but the general rigorous connection between the symplectic decomposition of the flow and the equations of motion, as obtained in our Lemma \ref{L:ODEs}, was not obtained in \cite{HPZ, HZ1, HZ2}.}

It then remains to show that $\|w(t)\|_{H_x^1} \lesssim h^2$ on the time scale $O(h^{-1})$, which we would like to prove using a suitable adaptation of the Lyapunov functional method initiated into the theory of orbital stability of single solitons by Weinstein \cite{W}.  Unfortunately, the presence of the $\Pi_z$ projection in \eqref{E:intro10} corrupts this computation and only yields a bound $\|w(t)\|_{H_x^1} \lesssim h$.  To eliminate this problem, we construct a function $\nu_z=O(h^2)$, whose only time dependence is through the parameter $z$, such that the distorted double-soliton function $\tilde u_z = u_z + \nu_z$ satisfies
\begin{equation}
\label{E:intro11}
\partial_t \tilde u_z = JH'(\tilde u_z) + O(h^4+\|w\|_{H^1}^2)
\end{equation}
which is just \eqref{E:intro10} without the $\Pi_z$ projection.  The construction of $\nu_z$ is carried out in \S \ref{S:approx-sol}.

We add this correction $\tilde v_z$ to our soliton manifold $M$ and consider the distorted manifold $\tilde M = \{ \, \tilde u_z \, \}$.  The solution $u$ to \eqref{E:NLS} now has a decomposition $u=\tilde u_z + \tilde w$ where $\tilde u_z$ satisfies \eqref{E:intro11} and it suffices to prove that $\|\tilde w(t) \|_{H_x^1} \lesssim h^2$.  In other words, we would like to show that the exact solution to \eqref{E:NLS} is approximately equal to the solution to the approximate equation \eqref{E:intro11}.  In \S \ref{S:Lyapunov}, a Lyapunov functional is employed to obtain the needed control on $\tilde w$.  The Lyapunov functional used is a superposition of two copies--one for each soliton--of the classical functional, built from energy, momentum and mass, employed by Weinstein \cite{W} to prove orbital stability of single solitons.  This superposition was previously used by Martel-Merle-Tsai \cite{MMT} in their study of the orbital stability of spreading multiple solitons.  Our presentation of this component of the argument is a little different from \cite{W} or \cite{MMT} and more in line with the abstract orbital stability theory developed by Grillakis-Shatah-Strauss \cite{GSS1, GSS2}.  Roughly, we prove that if $W:L^2\to \mathbb{R}$ is a (densely defined) functional such that the derivative is of order $O(h)$ on $\tilde M$, and if we define $L$ to be the quadratic part of $W$ above $\tilde M$, then $\partial_t L$ is essentially the quadratic part of the Poisson bracket $\{ H, W\}(u)$ above $\tilde M$, which we show is of order $O(h^5)$.  

Let us note that $h^{-\delta}$ losses occur in several estimates, which were not necessarily indicated in the above introduction, owing to the fact that in the attractive case $|v|$ can exceed $h$ by a factor of $\log h^{-1}$ and $a$ decreases below $a_0$, as well as the presence of an $x$-multiplication factor in terms involving $\partial_{v_j} u_z$ in both the attractive and repulsive cases.  We indicate the presence of such losses by writing, for example, $h^{4-}$.  These losses are more carefully quantified in the concluding summary of the proof in \S \ref{S:conclusion}.

We emphasize that the methods in \S \ref{S:eff-dyn} and \S \ref{S:Lyapunov}, although stated only for the problem at hand, are fairly general and widely applicable to problems in orbital stability of single \cite{W, GSS1, GSS2} and multiple \cite{MS} solitons and the dynamics of solitons in slowly varying potentials \cite{FGJS, HPZ, HZ2, DV, Munoz}, weak rough potentials \cite{AS, HZ1, P}, and the interaction of two soliton tails, as considered here.  
The portion of the analysis most specific to the problem at hand appears in \S \ref{S:approx-sol}, where the approximate solution is constructed. In this section, we consider the two components of the double-soliton separately and exploit the group structure of each individual soliton to pull-back to a nearly-stationary problem, which can be solved by operator inversion.  This method was introduced by Holmer-Zworski \cite{HZ2} to produce an improvement of the result by Fr\"ohlich-Gustafson-Jonsson-Sigal \cite{FGJS} on the dynamics of single solitons in a slowly-varying potential, eliminating the uncontrollable errors in the ODEs appearing in \cite{FGJS}.

Let us point out some related papers.  Marzuola-Weinstein \cite{MW} consider the dynamics of symmetric and antisymmetric states in a double well-potential.  Krieger-Martel-Rapha\"el \cite{KMR} construct two-soliton solutions with separating components asymptotically as $t\to +\infty$ for the nonlinear Hartree equation, where the long-range effects of the nonlinearity complicate the analysis but also lead to nonnegligible perturbations of the asymptotic trajectories.  Our analysis is similar in several ways to that of \cite{KMR}, although our priorties are different -- we study the dynamics for a finite (but dynamically significant) time of an initial data that is close to a double-soliton, wheras they provide infinite time dynamics for an exact double-soliton solution.  The problem of stability of nonintegrable NLS multiple solitons, with components that separate as $t\to \infty$, has been considered by Perelman \cite{Per}, Rodnianski-Soffer-Schlag \cite{RSS}, and Martel-Merle-Tsai \cite{MMT}.

We now remark on where we rely upon the even/odd symmetry assumption on the solution.  While the arguments in \S \ref{S:eff-dyn} yielding \eqref{E:V-20} apply in general, in \S \ref{S:approx-sol}, when constructing the solution $\tilde u_z$ to the approximate equation \eqref{E:intro11}, we do make use of the symmetry assumption, although we have sketched an argument (not included in this paper) showing how one can adapt the argument to the general case.  The symmetry assumption also greatly simplifies the computations carried out in Appendix A which ultimately yield the ODEs \eqref{E:aveqn}, \eqref{E:V-6}, \eqref{E:V-7} in Theorem \ref{T:main1}.  The integrals in the general case appear very complicated, and we are less confident that we could control the propagation of $O(h^4)$ errors, as previously remarked.  However, the one place where the symmetry assumption is used critically is in obtaining the \emph{upper} bound on the Lyapunov function used in \S \ref{S:Lyapunov} to show the closeness of the true solution $u$ and the solution $\tilde u_z$ of the approximate equation \eqref{E:intro11}.  Our guess is that to resolve this issue, one would need to restructure the Martel-Merle-Tsai Lyapunov function in a substantial way.  The lower bound on the Lyapunov function, however, carries through in general.

\subsection{Acknowledgements}
We thank Maciej Zworski and Galina Perelman for helpful discussion related to this paper.  J.H. was partially supported by NSF Grant DMS-0901582 and a Sloan Research Fellowship.

\section{Background on solitons, Hamiltonian structure, and Lyapunov functionals}
\label{S:background}

The NLS equation \eqref{E:NLS} can be put into Hamiltonian form as follows. 
Take as the ambient symplectic manifold $L^2=L^2(\mathbb{R};\mathbb{C})$ with metric  $\la v_1,v_2\ra_u = \Re \int v_1 \bar v_2$ for $u\in L^2$, $v_1,v_2\in T_uL^2=L^2$.  Let $J=-i$, viewed as an operator $T_uL^2\to T_uL^2$.  The corresponding symplectic form is $\omega_u(v_1,v_2) = \la v_1, J^{-1}v_2\ra_u$ (we henceforth drop the $u$-subscript).  Define the (densely defined, with domain $D=H^2$) Hamiltonian $H:L^2\to \mathbb{R}$ as
$$H(u) = \frac14 \int |u_x|^2 dx - \frac14 \int |u|^4 \,.$$
Using the metric $\la \cdot, \cdot \ra$ defined above, $H'(u)\in T^*_u L^2$ is identified with an element of $T_uL^2$.
The free NLS equation \eqref{E:NLS} is 
\begin{equation}
\label{E:NLS-Hform}
\partial_t u = JH'(u)
\end{equation}
Solutions to \eqref{E:NLS} also satisfy conservation of mass $M(u)$ and momentum $P(u)$, where
$$M(u) = \frac12 \int |u|^2 \,, \qquad P(u) = \frac12 \Im \int \bar u \, u_x \,.$$
Let $\phi(x) = \sech x$ and
$$\eta(x,\mu,a,\theta,v) = e^{i\theta}e^{i\mu^{-1}v(x-a)} \mu \phi(\mu (x-a))$$
Direct computation shows that $M(\eta) = \mu$ and $P(\eta)=v$.
Consider the manifold of solitons
$$M = \{ \, \eta(\cdot, \mu,a,\theta,v) \, | \, \,\,\mu>0, \theta\in \mathbb{R}, a\in \mathbb{R}, v\in \mathbb{R} \, \} \,.$$

The tangent space at $\eta=\eta(\cdot, \mu,a,\theta,v)$ is
$$T_{(\mu,a,\theta,v)} M = \spn \{ \, \partial_\mu \eta, \partial_\theta \eta, \partial_a \eta, \partial_v \eta \, \}\,.$$

Note that $JH'(\eta) \in T_{(\mu,a,\theta,v)} M$, and thus the flow associated to \eqref{E:NLS} will remain on $M$ if it is initially on $M$.  Specifically, direct computation shows
\begin{equation}
\label{E:JH0-prime}
JH'(\eta) = (\frac12 \mu^{-2} v^2 + \frac12 \mu^2) \partial_\theta \eta + \mu^{-1}v \partial_a \eta \,.
\end{equation}
To gain a better understanding of \eqref{E:free-flow} and \eqref{E:JH0-prime}, we can restrict $\omega$ to $M$ to obtain
$$i^*\omega = d\theta\wedge d\mu + dv\wedge da \,,$$
where $i:M\to L^2$ denotes the inclusion
and restrict $H$ to $M$ to obtain
$$H(\eta) = \frac12 \mu^{-1}v^2 - \frac16\mu^3 \,,$$
and then note that \eqref{E:free-flow} is just the solution to the Hamilton equations of motion for $H(\eta)$ with respect to $i^*\omega$:
\begin{equation}
\label{E:free-ODEs}
\left\{
\begin{aligned}
&\dot \mu = \partial_\theta H(\eta) =0 \\
&\dot a = \partial_v H(\eta) = \mu^{-1}v\\
& \dot \theta = - \partial_\mu H(\eta) = \frac12 \mu^{-2}v^2 + \tfrac12 \mu^2 \\
&\dot v = -\partial_a H(\eta) =0
\end{aligned}
\right.
\end{equation}
Suppose we knew that $JH'(\eta)\in T_{(\mu,a,\theta,v)}M$ and wanted to recover the coefficients as in \eqref{E:JH0-prime}. This could be achieved by noting that
\begin{align*}
JH'(\eta) &= \partial_t \eta \\
&= \dot \mu \partial_\mu \eta +\dot a \partial_a\eta + \dot \theta \partial_\theta \eta + \dot v \partial_v \eta \\
&= \partial_v H(\eta) \partial_a \eta - \partial_\mu H(\eta) \partial_\theta \eta
\end{align*}
Moreover, the functionals $M$ and $P$, considered as \emph{auxiliary} Hamiltonians, have associated Hamilton vector fields
$$JM'(\eta) = -\partial_\theta \eta \qquad JP'(\eta) = \partial_a \eta \,.$$
This enables us to write
\begin{equation}
\label{E:JH0-prime-2}
JH'(\eta) = \partial_v H(\eta) \, JP'(\eta) + \partial_\mu H(\eta) JM'(\eta) \,.
\end{equation}
From this, we learn that $W_{(\mu,a,\theta,v)}'(\eta) =0$, where
\begin{equation}
\label{E:W-classical}
W_{(\mu,a,\theta,v)}(u) \defeq -\partial_\mu H(\eta) M(u) - \partial_v H(\eta) P(u) + H(u) \,.
\end{equation}
The functional $L_{(\mu,a,\theta,v)}(u)=W_{(\mu,a,\theta,v)}(u)-W_{(\mu,a,\theta,v)}(\eta)$ is the Lyapunov functional used in the classical orbital stability theory for \eqref{E:NLS} due to Weinstein \cite{W}.

\section{Effective dynamics}
\label{S:eff-dyn}

Now we turn to the double soliton problem and begin the proof of Theorem \ref{T:main1}.  
Consider the two-soliton submanifold $M$ of $L^2$ given by 
$$M = \{ u_z \defeq \eta(\cdot, \mu_1, a_1, \theta_1, v_1) + \eta(\cdot, \mu_2, a_2, \theta_2, v_2) \, \} \,.$$
Note that $M$ is just the linear superposition of two single solitons.  We adopt the notation 
$$z=(z^1,z^2,z^3,z^4,z^5,z^6,z^7,z^8)= (\mu_1, a_1, \mu_2, a_2, \theta_1, v_1, \theta_2, v_2) \,,$$
for coordinates on this manifold $M$.    
Next, we give the form of the symplectic orthogonal projection operator 
$$\Pi_z: T_{u_z}L^2 \to T_zM\,,$$  
Note that $T_{u_z}L^2$ is naturally identified with $L^2$.  A consequence of the requirement that $\la f- \Pi_z f, J^{-1}\partial_{z^\ell} u_z\ra =0$, $\ell=1,\ldots, 8$ is that 
\begin{equation}
\label{E:proj1}
\Pi_z f = \sum_{\ell, m=1}^8  \la f, J^{-1}\partial_{z^\ell} u_z \ra a^{\ell m}(z) \partial_{z^m}u_z
\end{equation}
where 
$A(z) = (a_{\ell m}(z))$ is the $8\times 8$ matrix with components $a_{\ell m}(z) = \la \partial_{z^\ell} u_z, J^{-1}\partial_{z^m} u_z \ra$ and $A(z)^{-1}= (a^{\ell m}(z))$ is the inverse matrix.  

Let $i:M\to L^2$ denote the inclusion.  It follows from the definition of $A(z)$ that the restricted symplectic form $i^*(\omega)$ takes the form
\begin{equation}
\label{E:restricted-form}
i^*(\omega) = \frac12 \sum_{\ell,m=1}^8 a_{\ell m} dz^\ell \wedge dz^m \,.
\end{equation}
It also follows by substitution into \eqref{E:proj1} that
$$\Pi_zJH'(u_z) = -\sum_{\ell,m=1}^8 \partial_{z^\ell}H(u_z) \,  a^{\ell m}(z) \, \partial_{z^m} u_z$$
Consequently, the equation $\partial_t u_z = \Pi_zJH'(u_z)$ is equivalent to the system of equations 
$$\dot z^m = -\sum_{\ell=1}^8 \partial_{z^\ell} H(u_z) \, a^{\ell m}(z) \, \qquad m=1,\ldots, 8 \,,$$
which are precisely Hamiltonian's equations of motion for the restricted (to $M$) Hamiltonian $z\mapsto H(u_z)$ with respect to the restricted (to $M$) symplectic form $i^*(\omega)$.

We propose to model the solution $u$ to \eqref{E:NLS-Hform} by
\begin{equation}
\label{E:decomp}
u = u_z + w
\end{equation}
where $u_z\in M$ is chosen so that the symplectic orthogonality conditions
\begin{equation}
\label{E:orth}
\la w, J^{-1}\partial_{z^\ell} u_z\ra =0 \,, \qquad \ell=1,\ldots, 8.
\end{equation}
hold.  The fact that such a $z$ exists follows from the implicit function theorem and the assumed smallness of $w$.  Note that if we assume $u(t)$ solves \eqref{E:NLS-Hform}, this induces time dependence on the parameters $z \in \mathbb{R}^8$. \footnote{Note that here $w$ is properly understood as an element of $T_{u_z}L^2$ and in \eqref{E:decomp} we mean that, starting at $u_z$ we take the flow-forward (by ``time'' 1) in the direction $w$.  However, using the natural identification between $T_{u_z}L^2$ and $L^2$, \eqref{E:decomp} makes sense as an equation involving functions in $L^2$.}

\begin{lemma}[effective dynamics]
\label{L:ODEs}
Suppose that $u$ evolves according to \eqref{E:NLS-Hform} and $z$, $w$ are defined by \eqref{E:decomp} so that the orthogonality conditions \eqref{E:orth} hold.  Then 
\begin{equation}
\label{E:eff-dyn}
\| \partial_t u_z - \Pi_z JH'(u_z) \|_{T_zM} \lesssim  \|w\|_{H^1}^2 + \max_{1\leq n \leq 8} \|J^{-1}\partial_{z^n}\Pi_z^\perp JH'(u_z)\|_{H^1}^2\,.
\end{equation}
Equivalently, considering $M$ as an 8-dimensional symplectic manifold equipped with the symplectic form  $i^*(\omega)$ given in \eqref{E:restricted-form}, the Hamilton's equations of motion for $z$ induced by the restricted Hamiltonian $z\mapsto H(u_z)$ approximately hold as follows:
\begin{equation}
\label{E:eff-dyn2}
\left|\dot z^m +\sum_{\ell=1}^8 \partial_{z^\ell} H(u_z) \, a^{\ell m}(z)\right| \lesssim \|w\|_{H^1}^2+\max_{1\leq n \leq 8} \|J^{-1}\partial_{z^n}\Pi_z^\perp JH'(u_z)\|_{H^1}^2
\, \qquad m=1,\ldots, 8
\end{equation}
\end{lemma}
The norm $\|\cdot \|_{T_zM}$ is the one induced by the metric $\la \cdot, \cdot \ra_{u_z}$.  As $T_zM$ is finite-dimensional, we have the norm-equivalence to
$$\left\| \sum_{\ell=1}^8 \gamma(z^\ell) \partial_{z^\ell}u_z \right\|_{T_zM} \sim  \max_{1\leq \ell \leq 8} |\gamma(z^\ell)|$$
\begin{proof}
Since $u$ solves \eqref{E:NLS-Hform}, we obtain from \eqref{E:decomp} the equation for $w$:
\begin{equation}
\label{E:w-eqn-1}
\partial_t w = -(\partial_tu_z - \Pi_zJH'(u_z)) + \Pi_z^\perp JH'(u_z) + JH''(u_z)w + O_{H^1}(\|w\|_{H^1}^2)
\end{equation}
By applying $\partial_t$ to \eqref{E:orth}, we obtain
$$0 = \la \partial_t w , J^{-1}\partial_{z^n} u_z \ra + \la w, J^{-1}\partial_{z^n}\partial_t u_z \ra$$
Here we have used that $\partial_t \partial_{z^n} = \partial_{z^n} \partial_t$, which holds provided we adopt the convention that $\partial_{z^\ell} \dot z^m =0$ for all $1\leq \ell, m \leq 8$.  Substituting \eqref{E:w-eqn-1} and using that $\la \Pi_z^\perp JH'(u_z), J^{-1} \partial_{z^n} u_z \ra =0$, we obtain
\begin{equation}
\label{E:20}
0 = \text{A}+\text{B}+\text{C}+\text{D}
\end{equation}
where
\begin{align*}
&\text{A} = -\la \partial_t u_z - \Pi_z JH'(u_z), J^{-1}\partial_{z^n} u_z \ra \\
&\text{B} =  \la JH''(u_z)w, J^{-1}\partial_{z^n} u_z \ra \\
&\text{C} =  \la w, J^{-1} \partial_{z^n} \partial_t u_z \ra \\
&\text{D} =  \la O(w^2), J^{-1}\partial_{z^n} u_z \ra
\end{align*}
Since $J^*J^{-1}=-1$ and $H''(u_z)$ is self-adjoint,
$$\text{B} = -\la w, H''(u_z)\partial_{z^n} u_z \ra = -\la w, \partial_{z^n} H'(u_z) \ra = -\la w, J^{-1}\partial_{z^n} JH'(u_z) \ra$$
Hence
\begin{align*}
\text{B}+\text{C} &= \la w, J^{-1}\partial_{z^n} (\partial_t u_z -JH'(u_z)) \ra \\
&= \la w, J^{-1}\partial_{z^n} (\partial_t u_z -\Pi_zJH'(u_z)) \ra -\la w, J^{-1}\partial_{z^n} \Pi_z^\perp JH'(u_z) \ra 
\end{align*}
Let $R= \partial_t u_z - \Pi_z JH'(u_z) \in T_zM$, and expand with respect to the basis of $T_zM$ as
$$R= \sum_{\ell=1}^8 \gamma_\ell(z) \partial_{z^\ell} u_z \,.$$
It follows from \eqref{E:20} that
\begin{equation}
\label{E:21}
\la R, J^{-1} \partial_{z^n} u_z \ra = \la w, J^{-1}\partial_{z^n} R \ra - \la w, J^{-1}\partial_{z^n} \Pi_z^\perp JH'(u_z)\ra + O(\|w\|_{H^1}^2) \,.
\end{equation}
We have
$$\partial_{z^n} R = \sum_{\ell=1}^8 \partial_{z^n}\gamma_\ell(z) \, \partial_{z^\ell} u_z + \sum_{\ell=1}^8 \gamma_\ell(z) \partial_{z^n}\partial_{z^\ell} u_z \,.$$
Since $w\in T_zM^\perp$,
$$\la w, J^{-1}\partial_{z^n} R \ra = \sum_{\ell=1}^8 \gamma_\ell(z) \la w, J^{-1}\partial_{z^n}\partial_{z^\ell} u_z \ra$$
and hence
\begin{equation}
\label{E:22}
| \la w, J^{-1}\partial_{z^n} R \ra| \lesssim \|w\|_{H^1} \| R \|_{T_zM} \,.
\end{equation}
The lemma follows from \eqref{E:21}, $\|R\|_{T_zM} = \max_{1\leq n\leq 8} |\la R, J^{-1} \partial_{z^n} u_z\ra|$, \eqref{E:22}, and Cauchy-Schwarz.
\end{proof}

In our case we shall have 
$$\|J^{-1}\partial_{z^n} \Pi_z^\perp JH'(u_z)\|_{H^1}^2\lesssim h^{4-} \,.$$
We carry out computations of \eqref{E:eff-dyn} in Appendix \ref{A:computations} and show that \eqref{E:eff-dyn} is equivalent to \eqref{E:V-3}, with error terms $O(h^{4-})$, even without the even/odd assumption on the solution.  It is further shown in Appendix \ref{A:computations} that when the even/odd assumption is imposed and the integrals in \eqref{E:V-3} are explicitly computed, we obtain
$$
\left\{
\begin{aligned}
&\dot \mu = (-1)^\sigma (8a-4)ve^{-2a} + O(h^{4-})\\
&\dot a = \mu^{-1}v + (-1)^\sigma (-4a+\frac23\pi^2)ve^{-2a}+O(h^{4-})\\
&\dot \theta = \frac12 \mu^2 + \frac12 v^2 \mu^{-2} + 18(-1)^\sigma e^{-2a}+O(h^{4-})\\
&\dot v=-4(-1)^\sigma e^{-2a} + O(h^{4-})
\end{aligned}
\right.
$$
The solution $(\mu,a,\theta,v)$ is adequately approximated by the ODEs appearing in the statement of Theorem \ref{T:main1}.

\section{Approximate solution}
\label{S:approx-sol}

By Lemma \ref{L:ODEs} and \eqref{E:w-eqn-1}, the equation for $w$ is
$$
\left\{
\begin{aligned}
&\partial_t w = JH''(u_z)w + \Pi_z^\perp JH'(u_z) + O_{H^1}(\|w\|_{H^1}^2+h^{4-}) \\
&w\big|_{t=0}=w_0
\end{aligned}
\right.
$$

The next step is to show that there exists a function $\nu_z(x)$ such that $\|\nu_z\|_{H^1} \lesssim h^2$, whose only time dependence occurs through the parameter $z$, such that
\begin{equation}
\label{E:approx-sol}
\partial_t \nu_z = JH''(u_z)\nu_z + \Pi_z^\perp JH'(u_z)+ O_{H^1}(h^{4-}+\|w\|_{H^1}^2)
\end{equation}
 Here it is assumed that $z\in \mathbb{R}^8$ evolves according to Lemma \ref{L:ODEs}, i.e. 
\begin{equation}
\label{E:ODEs-3}
\partial_t u_z = \Pi_zJH'(u_z) + O_{H^1}(h^{4-}+\|w\|_{H^1}^2)
\end{equation}
The initial data $\nu_z\big|_{t=0}$ is \emph{not} prescribed but our structural assumption on $\nu_z$ is fairly rigid.  Note that given \eqref{E:ODEs-3}, the assertion that $\nu_z$ solve \eqref{E:approx-sol} is equivalent to the statement that $\tilde u_z \defeq u_z+\nu_z$ solve
\begin{equation}
\label{E:approx-eqn}
\partial_t \tilde u_z = JH'(\tilde u_z)+O_{H^1}(h^{4-}+\|w\|_{H^1}^2) \,.
\end{equation}
This is an approximate solution to \eqref{E:NLS-Hform} (that does not, in general, satisfy the specified initial data).  

Let $g:L^2\to L^2$ be the operator attached to the parameters $(\mu,a,\theta,v)$ that acts on a function $\rho$ as follows:
\begin{equation}
\label{E:corr240}
(g \rho) (x) = e^{i\theta}e^{i\mu^{-1}v(x-a)}\mu \rho(\mu(x-a)) \,. 
\end{equation}
The inverse action is
$$g^{-1}\rho(x) = e^{-i\theta}e^{-i\mu^{-2}vx} \mu^{-1} \rho(\mu^{-1}x+a) \,.$$
The adjoint action $g^*$ with respect to $\la \cdot, \cdot\ra$ is
$$g^* \rho(x) = e^{-i\theta} e^{-i\mu^{-2}vx} \rho( \mu^{-1}x+a) = \mu g^{-1}\rho (x) \,.$$
Denote $\phi(x)=\sech x$.  Then $u_z = g_1\phi+g_2\phi$.  We look for a solution $\nu_z$ to \eqref{E:approx-sol} in the form
\begin{equation}
\label{E:C2}
\nu_z= \sum_{j=1}^2 \alpha_j \; g_j\rho_j \,,
\end{equation}
where $\alpha_j=\alpha_j(\mu_1,a_1,\theta_1,v_1,\mu_2,a_2,\theta_2,v_2)$ and  $g_j$ is the operator corresponding to $(\mu_j,a_j,\theta_j,v_j)$.   That is, we assume $\nu_z$ can be decomposed into two pieces, each of which can be pulled back to a stationary equation and solved by operator inversion.  The time dependence of $\nu_z$ occurs only through $z$.

The next step is to substitute \eqref{E:C2} into \eqref{E:approx-sol}.  The resulting equation simplifies provided we assume that each $\rho_j$ satisfies $|\rho_j(x)| \lesssim h^2 e^{(-1+)|x|}$ for $|x|\geq 1$ as certain cross terms become $O_{H^1}(h^{4-})$.  In this case, \eqref{E:approx-sol} will be satisfied provided for both $j=1,2$, we have
$$\partial_t(\alpha_j g_j\rho_j) = JH''(g_j\phi)(\alpha_jg_j\rho_j) + \Pi_z^\perp J((g_j\phi)^2\overline{g_{3-j}\phi} + 2|g_j\phi|^2g_{3-j}\phi) + O_{H^1}(h^{4-}+\|w\|_{H^1}^2)$$
In the proof, we delete the $j$-subscripts, denote $\tilde g =g_{3-j}$ and moreover assume that $\dot \alpha_j= O(h^2)$.  Then we aim to solve
$$\partial_t(g\rho) = JH''(g\phi)(g\rho) + \alpha^{-1} \Pi^\perp J((g\phi)^2\overline{\tilde g\phi} + 2|g\phi|^2\tilde g\phi) +O_{H^1}(h^{4-}+\|w\|_{H^1}^2)$$
The form of the operator $\Pi^\perp$ can be simplified, since we only need to keep the $O(1)$ and $O(h)$ parts.  This equation takes the form
\begin{equation}
\label{E:corr1}
\partial_t (g\rho) = JH''(g\phi) (g\rho) +\alpha^{-1} \Pi^\perp Jgf +O_{H^1}(h^{4-}+\|w\|_{H^1}^2)
\end{equation}
where, in the case $j=1$,
\begin{equation}
\label{E:corr201}
f
\begin{aligned}[t]
&= g_1^{-1}[(g_1\phi)^2\overline{g_2\phi}] + 2 g_1^{-1}[|g_1\phi|^2 g_2\phi] \\
&= 
\begin{aligned}[t]
& e^{i\omega_1} e^{i\omega_2 x} \mu_j^2 \mu_{3-j} \phi(\mu_{3-j}\mu_j^{-1}x+(a_j-a_{3-j})) \phi(x)^2 \\
&+ 2 e^{i\omega_3} e^{i\omega_4x} \mu_j^2 \mu_{3-j} \phi(\mu_{3-j}\mu_j^{-1}x + (a_j-a_{3-j})) \phi(x)^2
\end{aligned}
\end{aligned}
\end{equation}
with
\begin{align*}
&\omega_1 \defeq \theta_1-\theta_2 - \mu_{3-j}^{-1}v_{3-j}(a_j-a_{3-j}) \\
&\omega_2 \defeq \mu_j^{-2}v_j - \mu_{3-j}^{-1}\mu_j^{-1} v_{3-j} \\
&\omega_3 \defeq \theta_2-\theta_1 + \mu_{3-j}^{-1}v_{3-j}(a_j-a_{3-j}) \\
&\omega_4 \defeq \mu_j^{-2}v_j + \mu_{3-j}^{-1}\mu_j^{-1}v_{3-j} 
\end{align*}
In the case $j=2$,
\begin{equation}
\label{E:corr202}
f= g_2^{-1}[(g_2\phi)^2\overline{g_1\phi}] + 2 g_2^{-1}[|g_2\phi|^2 g_1\phi] \,.
\end{equation}
with a similar expansion.
The only important feature of these expressions is that $e^{(1-)\la x  \ra}f = O(h^2)$ and $e^{(1-)\la x \ra}\partial_t f = O(h^3)$ when $\theta_1-\theta_2$ is a constant (as in the case of the even/odd symmetry assumption in Theorem \ref{T:main1}.

Now we begin the task of pulling back \eqref{E:corr1} --
applying $g^*$ to \eqref{E:corr1}, we obtain
\begin{equation}
\label{E:C5}
g^*\partial_t g\rho = g^*[JH''(g\phi)(g\rho)] + \alpha^{-1} g^*\Pi^\perp Jgf + O_{H^1}(h^{4-}+\|w\|_{H^1}^2) \,.
\end{equation}

First, we aim to simplify the term $g^*[JH''(g\phi)(g\rho)]$ in \eqref{E:C5}. 
Let $K_g(\phi) = H(g\phi)$.   It follows that $K_g'(\phi) = g^*H'(g\phi)$ and $K_g''(\phi) = g^*[H''(g\phi)(g\bullet)]$.  By direct substitution, we compute:
$$K_g(u) = \mu^3 H(u) + \mu v P(u) + \frac12v^2\mu^{-1} M(u)$$
Since $g^*J = Jg^*$, we have
\begin{equation}
\label{E:C6}
\begin{aligned}
g^*[JH''(g\phi)(g\bullet)] &= JK''_g(\phi)\\
&= \mu^3JH''(\phi) + \mu v JP''(\phi) + \tfrac12 v^2\mu^{-1} JM''(\phi)
\end{aligned}
\end{equation}

Second, we seek to simplify the term $g^*\partial_tg\rho$ in \eqref{E:C5}.
Define the operators
\begin{align*}
&\tilde \partial_\mu \defeq \partial_\mu\big|_{(1,0,0,0)} =  \partial_x x \\
&\tilde \partial_a \defeq \partial_a \big|_{(1,0,0,0)} = - \partial_x \\
&\tilde \partial_\theta \defeq   \partial_\theta \big|_{(1,0,0,0)} = i \\
&\tilde \partial_v \defeq  \partial_v\big|_{(1,0,0,0)}=  ix
\end{align*}

Let
\begin{equation}
\label{E:C4}
\begin{aligned}
&\bar \partial_\mu \defeq g^*\partial_\mu g 
&&= \partial_x x - i\mu^{-2}vx 
&&= \tilde\partial_\mu -\mu^{-2}v\tilde \partial_v \\
&\bar \partial_a \defeq  g^*\partial_a g 
&&= - \mu^2 \partial_x -iv 
&&=\mu^2\tilde\partial_a-v\tilde\partial_\theta  \\
&\bar \partial_\theta \defeq  g^*\partial_\theta g 
&&= i\mu 
&&=\mu\tilde\partial_\theta \\
&\bar \partial_v \defeq  g^*\partial_v g 
&&= i\mu^{-1}x 
&&=\mu^{-1}\tilde\partial_v\\
\end{aligned}
\end{equation}
It follows from the chain rule that
$$g^*\partial_t g = \dot \mu \bar\partial_\mu + \dot a \bar \partial_a + \dot \theta \bar \partial_\theta + \dot v \bar \partial_v \,.$$
Using \eqref{E:eff-dyn},
\begin{equation}
\label{E:C7}
\begin{aligned}
g^*\partial_t g &= v\mu^{-1} \bar\partial_a + \tfrac12\mu^2 \bar \partial_\theta +O_{H^1}(h^2)\\
&= v\mu \tilde\partial_a + \tfrac12 \mu^3 \tilde\partial_\theta + O_{H^1}(h^2)
\end{aligned}
\end{equation}

Finally, we aim to simplify the term $g^*\Pi^\perp Jgf$ in \eqref{E:C5}.  We will show that
\begin{equation}
\label{E:C8}
g^* \Pi_{(\mu,a,\theta,v)} Jgf = \mu \Pi_{(1,0,0,0)} Jf
\end{equation}

Using that $J^*J^{-1}=-1$ and \eqref{E:C4}, we obtain
\begin{align*}
g^* \Pi J gf 
&= 
\begin{aligned}[t]
&\la J gf, J^{-1}\partial_a g \phi\ra g^*\partial_v g \phi 
- \la Jgf, J^{-1}\partial_v g \phi \ra g^*\partial_a g \phi \\
&+\la Jgf, J^{-1} \partial_\mu g \phi\ra g^* \partial_\theta g\phi 
- \la Jgf, J^{-1} \partial_\theta g \phi \ra g^* \partial_\mu g \phi
\end{aligned}\\
&= 
- \la f, \bar \partial_a \phi \ra \bar \partial_v \phi 
+ \la f,  \bar \partial_v \phi \ra  \bar \partial_a \phi
- \la f, \bar \partial_\mu \phi \ra  \bar \partial_\theta \phi 
+ \la f,  \bar \partial_\theta \phi \ra \bar \partial_\mu \phi
\end{align*}
Substituting \eqref{E:C4}, after a few cancelations we obtain
\begin{align*}
\mu^{-1}g^* \Pi J gf 
&= 
- \la f, \tilde \partial_a \phi \ra \tilde \partial_v \phi 
+ \la f,  \tilde  \partial_v \phi \ra  \tilde \partial_a \phi
- \la f, \tilde \partial_\mu \phi \ra  \tilde \partial_\theta \phi 
+ \la f,  \tilde \partial_\theta \phi \ra \tilde \partial_\mu \phi \\
&=
\begin{aligned}[t]
&+ \la Jf, J^{-1}\tilde \partial_a \phi \ra \tilde \partial_v \phi 
- \la Jf,  J^{-1}\tilde  \partial_v \phi \ra  \tilde \partial_a \phi\\
&+ \la Jf, J^{-1}\tilde \partial_\mu \phi \ra  \tilde \partial_\theta \phi 
- \la Jf,  J^{-1}\tilde \partial_\theta \phi \ra \tilde \partial_\mu \phi
\end{aligned}
\end{align*}
which establishes \eqref{E:C8}.

Note that it follows from \eqref{E:C8} that
$$g^*\Pi^\perp_{(\mu,a,\theta,v)} g Jf = \mu \Pi^\perp_{(1,0,0,0)} Jf$$

Using the expressions \eqref{E:C6}, \eqref{E:C7}, and \eqref{E:C8}, the equation \eqref{E:C5} converts to
$$v\mu \tilde\partial_a \rho + \tfrac12 \mu^3 \tilde\partial_\theta\rho + \mu \partial_t \rho = 
\begin{aligned}[t]
&\mu^3JH''(\phi)(\rho) + \mu v JP''(\phi)(\rho)  \\
&+ \alpha^{-1}\mu \Pi_{(1,0,0,0)}^\perp Jf + O_{H^1}(h^{4-}+\|w\|_{H^1}^2)
\end{aligned}
$$
Noting that $JP''(\phi) = \tilde\partial_a$ and $JM''(\phi) = -\tilde \partial_\theta$, the equation becomes
$$\tfrac12 \mu^3 JM''(\phi)\rho  + \mu^3JH''(\phi)\rho - \mu \partial_t \rho =  -\alpha^{-1}\mu \Pi_{(1,0,0,0)}^\perp Jf + O_{H^1}(h^{4-}+\|w\|_{H^1}^2)$$
Hence we see we should take $\alpha = \mu^{-2}$ so that the equation becomes
$$
\tfrac12  JM''(\phi)\rho + JH''(\phi)\rho - \mu^{-2}\partial_t \rho =  -\Pi_{(1,0,0,0)}^\perp Jf + O_{H^1}(h^{4-}+\|w\|_{H^1}^2)
$$
Now apply $J^{-1}$ to obtain the equation
\begin{equation}
\label{E:corr200}
S\rho = J^{-1}\mu^{-2}\partial_t \rho -J^{-1} \Pi_{(1,0,0,0)}^\perp Jf + O(h^{4-}+\|w\|_{H^1}^2)
\end{equation}
where the operator
$$S(\rho) \defeq \tfrac12 M''(\phi)(\rho) + H''(\phi)(\rho) = \tfrac12\rho - \tfrac12 \partial_x^2\rho - 2|\phi|^2\rho - \phi^2\bar\rho $$
is self-adjoint with respect to the inner product $\la u,v\ra = \Re
\int u\bar v$.  The kernel is spanned by $\tilde \partial_a \phi$
and $\tilde \partial_\theta \phi$.  

\begin{lemma}[properties of $S$] \quad
\label{L:S-prop}
\begin{enumerate}
\item For any $f\in H^1$, let $F = J^{-1} \Pi_{(1,0,0,0)}^\perp Jf$.  Then $F$ satisfies the orthogonality conditions
\begin{align}
\label{E:corr210}
&\la F, \tilde \partial_\theta \phi \ra =0 \,,
&\la F, \tilde \partial_a \phi \ra =0 \\
\label{E:corr211}
&\la F, \tilde \partial_\mu \phi \ra = 0 \,,
&\la F, \tilde \partial_v \phi \ra =0
\end{align}
\item For any $F$ satisfying \eqref{E:corr210}, $S^{-1}F$ is defined and satisfies the boundedness properties
\begin{align}
\label{E:corr212}
&\|S^{-1}F \|_{H^1} \lesssim \|F\|_{L^2}\,, \\
\label{E:corr219}
&\| e^{\sigma \la x \ra} S^{-1} F \|_{H^2} \lesssim \| e^{\sigma \la x \ra} F\|_{L^2} \,.
\end{align}
for $0\leq \sigma <1$.
\item  For any $F$ satisfying \eqref{E:corr210} and \eqref{E:corr211}, $S^{-1}F$ satisfies the orthogonality properties
\begin{align}
\label{E:corr213}
&\la S^{-1}F, \tilde \partial_\theta \phi \ra =0\,, 
&\la S^{-1}F, \tilde \partial_a \phi \ra =0 \\
\label{E:corr214}
&\la J^{-1}S^{-1}F, \tilde \partial_\theta \phi \ra =0\,, 
&\la J^{-1}S^{-1}F, \tilde \partial_a \phi \ra =0
\end{align}
\end{enumerate}
\end{lemma}
\begin{proof}
Item (1) is immediate from the definition of $\Pi_{(1,0,0,0)}$.
For item (2), we recall that $\ker S = \spn \{\tilde \partial_a \phi, \tilde \partial_\theta \phi \}$ and moreover, $0$ is an isolated point in the spectrum of $S$.  Thus $S^{-1}: (\ker S)^\perp \to (\ker S)^\perp$ is bounded as an operator on $L^2$.  The inequality \eqref{E:corr212} follows from this and elliptic regularity.  To prove \eqref{E:corr219}, it suffices to show that for any $G$ and any $|\sigma| <1$, we have
\begin{equation}
\label{E:corr220}
\|G\|_{H^2} \lesssim \|e^{\sigma x} S e^{-\sigma x} G \|_{L^2} + \|e^{-2|x|}G\|_{L^2}
\end{equation}
Indeed, \eqref{E:corr219} follows by taking $G=e^{\sigma x}S^{-1}F$, appealing to \eqref{E:corr212}, and separately considering $\sigma>0$ and $\sigma <0$ with $|\sigma|<1$.  To establish \eqref{E:corr220}, we calculate
\begin{equation}
\label{E:conj-S}
\begin{aligned}[t]
e^{\sigma x} S e^{-\sigma x} G  &= (S + \sigma\partial_x - \tfrac12\sigma^2)G \\
&= (\frac12(1-\sigma^2) +\sigma\partial_x  - \frac12 \partial_x^2 )G - 2\phi^2G - \phi^2 \bar G 
\end{aligned}
\end{equation}
and hence
$$(\frac12(1-\sigma^2) + \sigma \partial_x - \frac12\partial_x^2)G = e^{\sigma x}Se^{-\sigma x}G + 2\phi^2G + \phi^2 \bar G$$
On the left-hand side, we have an operator with symbol $\frac12(1-\sigma^2) + \sigma i\xi + \xi^2$, which dominates $\la \xi\ra^2$ under our assumption on $\sigma$.  From this and the fact that $|\phi(x)|^2\leq e^{-2|x|}$, we conclude \eqref{E:corr220}. 

For item (3), \eqref{E:corr213} follows from the fact that $S^{-1}: (\ker S)^\perp \to (\ker S)^\perp$.  To establish \eqref{E:corr214}, we note that by \eqref{E:corr211},
$$0 = \la F, \tilde \partial_\mu \phi \ra = \la S^{-1}F, S\tilde \partial_\mu \phi\ra$$
and similarly
$$0 = \la F, \tilde \partial_v \phi \ra = \la S^{-1}F, S\tilde \partial_v \phi \ra$$
and thus it suffices to establish that $S(\tilde \partial_v \phi) = J^{-1}\tilde \partial_a \phi$ and $S(\tilde \partial_\mu \phi)= J^{-1}\tilde \partial_\theta \phi$.  To prove these equalities, recall
\begin{equation}
\label{E:W-recall}
0=W'(\eta)=(\frac12 \mu^2+\frac12\mu^{-2}v^2)M'(\eta) - \mu^{-1}vP'(\eta) + H'(\eta)
\end{equation}
Taking $\partial_v$ and evaluating at $(\mu,a,\theta,v)=(1,0,0,0)$ gives  
$$S(\tilde \partial_v \phi) = P'(\phi)=J^{-1}\tilde \partial_a \phi \,.$$
Taking $\partial_\mu$ of \eqref{E:W-recall} and evaluating at $(\mu,a,\theta,v)=(1,0,0,0)$ gives  
$$S(\tilde \partial_\mu \phi) = -M'(\phi) = J^{-1}\tilde \partial_\theta \mu$$
\end{proof}

Recall that the task is to solve \eqref{E:corr200} where $f$ is either \eqref{E:corr201} or \eqref{E:corr202}.  At this point, we impose the even/odd solution assumption as in Theorem \ref{T:main1} which implies that $\theta_1-\theta_2$ is constant.  The other time dependent parameters in \eqref{E:corr201}, \eqref{E:corr202} are all slowly varying so that $\partial_t f_j = O(h^3)$.  Thus, we can solve \eqref{E:corr200} by iteration.
Let\footnote{As indicated earlier, $\rho$ can stand for either $\rho_1$ or $\rho_2$.  The superscript introduced here is different and meant to indicate part of an asymptotic expansion for either function.  In other words, we have $\rho_j = \rho_j^1+\rho_j^2$ for both $j=1,2$.}
\begin{equation}
\label{E:corr230}
\rho^1 = -S^{-1}J^{-1} \Pi_{(1,0,0,0)}^\perp J f
\end{equation}
By Lemma \ref{L:S-prop}(1)(2), this is well-defined with $\rho^1 = O(h^{2-})$ and satisfying all the needed regularity properties.  With $\rho^2$ as yet undefined, we plug $\rho^1+\rho^2$ into \eqref{E:corr200} to obtain
$$S\rho^2 = J^{-1}\mu^{-2}\partial_t\rho^1 + J^{-1}\mu^{-2}\partial_t\rho^2 + O(h^{4-}+\|w\|_{H^1}^2)$$
As mentioned previously, $\partial_t f = O(h^{3-})$ and thus
$$J^{-1}\mu^{-2}\partial_t\rho^1 = -J^{-1}\mu^{-2}S^{-1}J^{-1}\Pi_{(1,0,0,0)}^\perp J\partial_t f$$
is also $O(h^{3-})$.  By Lemma \ref{L:S-prop}(3), in particular  \eqref{E:corr214}, with $F=J^{-1}\Pi_{(1,0,0,0)}^\perp J\partial_t f$, we have that 
$$\tilde F \defeq J^{-1}\mu^{-2}\partial_t\rho^1 = -J^{-1}\mu^{-2}S^{-1}F$$
satisfies the condition \eqref{E:corr210}, and hence we can apply Lemma \ref{L:S-prop}(2) with $F$ replaced by $\tilde F$.  That is, the function
\begin{equation}
\label{E:corr231}
\rho^2 \defeq S^{-1} \tilde F = -S^{-1}J^{-1}\mu^{-2}S^{-1}J^{-1}\Pi_{(1,0,0,0)}^\perp J\partial_t f
\end{equation}
satisfies all the needed regularity properties.  Note further that $\partial_t \rho^2 =O(h^{4-})$.  Upon substituting $\rho^1+\rho^2$ into \eqref{E:corr200} with $\rho^1$ defined by \eqref{E:corr230} and $\rho^2$ defined by \eqref{E:corr231}, we find that equality holds with $O(h^{4-})$ error.  

Thus we have successfully constructed a solution to the approximate equation \eqref{E:approx-sol}.  We summarize our conclusions in the next lemma.

\begin{lemma}[approximate solution]
Recall the operator $g_j$ associated to $(\mu_j,a_j,\theta_j,v_j)$ defined in \eqref{E:corr240} and $f_j$ defined in \eqref{E:corr201} ($j=1$) or \eqref{E:corr202} ($j=2$).  Let $\rho_j^1$ be given by \eqref{E:corr230} and then let $\rho_j^2$ be given by \eqref{E:corr231}.  Then $\rho_j^k$ for $1\leq j,k\leq 2$ satisfy 
$$\|e^{(1-)\la x \ra} \rho_j^k \|_{H^2} \lesssim h^{1+k-} \,,$$
Let
$$\nu_z = \mu_1^2 g_1 (\rho_1^1+\rho_1^2) + \mu_2^2 g_2 (\rho_2^1+\rho_2^2)$$
Suppose that the parameter $z\in \mathbb{R}^8$ evolves according to the ODEs obtained from Lemma \ref{L:ODEs} (in the same phase or opposite phase case).  Then $\nu_z(x)$ solves \eqref{E:approx-sol}.
\end{lemma}

\section{Lyapunov functional}
\label{S:Lyapunov}

The final step is to show that the true solution $u$ to \eqref{E:NLS-Hform} is approximately the approximate solution $\tilde u_z = u_z+\nu_z$.  For this purpose, we introduce a Lyapunov functional.  First, some general considerations.  We consider the ``perturbed'' 8-dimensional manifold
$$\tilde M = \{ \, \tilde u_z \, | \, z\in \mathbb{R}^8 \, \}$$
Introduce the notation $\tilde w = u-\tilde u_z$ (so that $w= \tilde w + \nu_z$).  Now it follows from \eqref{E:approx-sol} that
\begin{equation}
\label{E:approx-sol-1}
\partial_t \tilde u_z = JH'(\tilde u_z)  + F \,.
\end{equation}
where $F = O_{H^1}(h^{4-}+\|\tilde w\|_{H^1}^2)$.

Suppose that $W_z:L^2\to \mathbb{R}$ is a densely defined functional.  We write $\partial_{z^\ell} W_z:L^2\to \mathbb{R}$ to indicate partial derivatives with respect to $z$ and 
$$W_z'(u) \in T^*_uL^2 \underset{\text{metric }g}{\simeq} T_uL^2 \simeq L^2$$ 
to indicate partial derivatives with respect to $u$ (ignoring the interdependence between $z$ and $u$ given by \eqref{E:decomp}, \eqref{E:orth}).  

Suppose that $W_z$ can be extended to a differentiable functional $H^1\to \mathbb{R}$; then for each $u\in H^1$, we have a bounded linear map $W_z'(u):H^1\to \mathbb{R}$ which, under the aforementioned identification, becomes a function belonging to $H^{-1}$.  In fact, our choice of $W_z$ is differentiable at all orders as a map $H^1\to \mathbb{R}$, which is to say that $W_z^{(k)}(u): \underbrace{H^1\times \cdots \times H^1}_{k \text{ copies}} \to \mathbb{R}$ is a bounded $k$-multilinear map.

We further assume that $\partial_{z^\ell} W_z(u) =0$ unless $|\dot z^\ell| \lesssim h$. Let
\begin{equation}
\label{E:L-def}
L_z(u) = W_z(u)-W_z(\tilde u_z) - \la W_z'(\tilde u_z), \tilde w \ra \,.
\end{equation}
That is, $L_z(u)$ is the quadratic part of $W_z(u)$ above the base manifold $\tilde M$.  
Now viewing $u=u(t)$ and $z=z(t)$ in accordance with \eqref{E:decomp}, \eqref{E:orth} (and thus reinstating the interdependence between $z$ and $u$), we have, for any functional $G_z:L^2\to \mathbb{R}$, 
$$\partial_t G_z(u) = \la G_z'(u), \partial_t u\ra + \sum_{k=1}^8 [\partial_{z^k}G](u)\dot z^k \,.$$
This leads to:

\begin{lemma}
\label{L:lyap}
Suppose that $u$ solves \eqref{E:NLS-Hform} and $z$ evolves so that $\tilde u_z$ solves \eqref{E:approx-sol-1}, and that $L_z(u)$ is given by \eqref{E:L-def}, the quadratic part of $W_z(u)$ above $\tilde M$.  Then
\begin{equation}
\label{E:lya4}
\partial_t L_z(u) = \{ H,W_z\}(u) - \{H, W_z\}(\tilde u_z) - \la \{H, W_z\}'(\tilde u_z), \tilde w\ra  - E_1+E_2\,,
\end{equation}
where
$$E_1 \defeq  \la W''(\tilde u_z) F, \tilde w\ra + \la W'(\tilde u_z), [JH'(u)-JH'(\tilde u_z) - JH''(\tilde u_z)\tilde w] \ra$$
and
$$E_2 \defeq \sum_{k=1}^8 \left( [\partial_{z^k}W_z](u) - [\partial_{z^k}W_z](\tilde u_z) - \la [\partial_{z^k}W_z'](\tilde u_z),\tilde w\ra \right) \dot z^k$$
In other words, $\partial_t L_z(u)$ is, up to error $E_1$ and $E_2$, the quadratic part of $\{H,W_z\}(u)$ above $\tilde M$.  Note that $E_2$ just involves the quadratic part of $[\partial_{z^k}W](u)$ above $\tilde M$.
\end{lemma}
In the typical application of this lemma (as for our $W_z$, defined below), we have bounded operators $W_z''(\tilde u_z):H^1\to H^{-1}$ and $W_z'''(\tilde u_z):H^1\times H^1 \to L^2$ which implies the bound
$$|E_1| \lesssim \|F\|_{H^1}\|\tilde w\|_{H^1} + \|W'(\tilde u_z)\|_{H^1}\|\tilde w\|_{H^1}^2$$
Thus, one just needs $\|F\|_{H^1} \lesssim h^3$ and $\|W'(\tilde u_z) \|_{H^1} \lesssim h$; in our case we in fact have the stronger statements $\|F\|_{H^1} \lesssim h^{4-}$ and $\|W'(\tilde u_z)\|_{H^1} \lesssim h^{2-}$.  Moreover, in our case we will have
$$|E_2|\lesssim h^{2-} \|\tilde w\|_{H^1}^2$$
since $\dot \mu, \dot v = O(h^{2-})$.

\begin{proof}
By \eqref{E:L-def},
$$\partial_t L_z(u) = \partial_t W_z(u) - \partial_t W_z(\tilde u_z) - \partial_t \la W_z'(\tilde u_z),\tilde w \ra$$
We compute each of the three terms on the right-hand side separately.
\begin{align}
\partial_tW_z(u) 
\notag &= \la W_z'(u), \partial_t u \ra + \sum_{k=1}^8 [\partial_{z^k}W_z](u) \dot z^k \\
\label{E:lya1}
&= \la W_z'(u), JH'(u) \ra + \sum_{k=1}^8 [\partial_{z^k}W_z](u) \dot z^k
\end{align}
where we invoked \eqref{E:NLS-Hform}.  Second, we compute
\begin{align}
\notag
\partial_tW_z(\tilde u_z) 
&= \la W_z'(\tilde u_z), \partial_t \tilde u_z \ra + \sum_{k=1}^8 [\partial_{z^k}W_z](\tilde u_z) \dot z^k \\
\label{E:lya2}
&= \la W_z'(\tilde u_z), JH'(\tilde u_z) \ra + \la W_z'(\tilde u_z), F \ra + \sum_{k=1}^8 [\partial_{z^k}W_z](\tilde u_z) \dot z^k
\end{align}
where we invoked \eqref{E:approx-sol-1}.  Finally, we compute
\begin{align}
\notag
\indentalign \partial_t \la W_z'(\tilde u_z), u-\tilde u_z\ra \\
\notag
&= 
\begin{aligned}[t]
&\la W_z''(\tilde u_z) \partial_t \tilde u_z, u - \tilde u_z \ra 
+ \la W_z'(\tilde u_z), \partial_t u - \partial_t \tilde u_z \ra \\
&+ \sum_{k=1}^8 \la [\partial_{z^k} W]'(\tilde u_z), u-\tilde u_z \ra \dot z^k 
\end{aligned}\\
\label{E:lya3}
&= 
\begin{aligned}[t]
&\la W_z''(\tilde u_z)JH'(\tilde u_z), \tilde w\ra + \la W''(\tilde u_z)F, \tilde w \ra + \la W_z'(\tilde u_z), JH'(u)-JH'(\tilde u_z)\ra \\
&- \la W'(\tilde u_z),F \ra + \la \sum_{k=1}^8 [\partial_{z^k}W]'(\tilde u_z), \tilde w\ra \dot z^k
\end{aligned}
\end{align}
Taking \eqref{E:lya1} minus \eqref{E:lya2} minus \eqref{E:lya3}, noting the cancelation of $+ \la W_z'(\tilde u_z), F \ra$ in \eqref{E:lya2} with $- \la W'(\tilde u_z),F \ra$ in \eqref{E:lya3}, we obtain \eqref{E:lya4}.
\end{proof}

To produce $W_z(u)$, we use an idea of Martel-Merle-Tsai \cite{MMT}. Let
$$\Psi(x) =
\begin{cases}
1 & \text{if } x \geq 1 \\
0 & \text{if } x \leq -1 \\
\end{cases}
$$
such that $(\Psi'(x))^2\lesssim \min(\Psi(x), 1-\Psi(x))$.  Set $\delta = 4/(\log h^{-1})=4/a_0$, so $0<\delta \ll 1$.  
Introduce the localizations $\psi_2(x) =
\Psi(\delta x)$ and $\psi_1(x) = 1- \Psi(\delta x)$, and
set $M_j(u) = M(\psi_j^{1/2} u)$ and $P_j(u) = P(\psi_j^{1/2} u)$. Define
\begin{equation}
\label{E:lya30}
\begin{aligned}[t]
W_z(u) &\defeq -\sum_{j=1}^2 \frac{\partial H(\eta_j)}{\partial
\mu_j} M_j(u) - \sum_{j=1}^2 \frac{\partial H(\eta_j)}{\partial
v_j} P_j(u) + H(u) \\
&=  \frac12\sum_{j=1}^2 (\mu_j^2+\mu_j^{-2}v_j^2) M_j(u) - \sum_{j=1}^2 \mu_j^{-1}v_j P_j(u) + H(u)
\end{aligned}
\end{equation}
The Lyapunov functional $L_z(u)$ we use is then defined as in \eqref{E:L-def}.

Lemma \ref{L:lyap} facilitates the computation of $\partial_t L_z(u)$, since $W_z(u)$ is built from ``nearly conserved'' quantities.  Indeed, we have the following Poisson brackets:
\begin{align*}
&\{ H,M_j \}(u) =  \frac12 \Im \int \psi_j' \; \bar u u_x \\
&\{H, P_j\}(u) =  \int \psi_j' (\frac12 |u_x|^2 - \frac14 |u|^4) - \frac18 \int \psi_j''' |u|^2
\end{align*}
It thus follows from Lemma \ref{L:lyap} that 
\begin{equation}
\label{E:lya34}
\begin{aligned}
\partial_t L_z(u) = 
&\frac12 \sum_{j=1}^2 (\mu_j^2+\mu_j^{-2}v_j^2) \la \{H,M_j\}''(\tilde u_z)\tilde w , \tilde w\ra- \sum_{j=1}^2 \mu_j^{-1}v_j \la \{H,P_j\}''(\tilde u_z)\tilde w, \tilde w\ra \\
&+O(\|w\|_{H^1}^3) - E_1 + E_2
\end{aligned}
\end{equation}
For our choice of $W_z(u)$, as remarked earlier, we have suitable bounds for $E_1$ and $E_2$.  Moreover, once one imposes the even/odd solution assumption of Theorem \ref{T:main1}, we have $\mu_1=\mu_2$ and $v_1^2=v_2^2$, so the first term in \eqref{E:lya34} disappears.\footnote{In fact, this is more easily seen by observing that once $\mu_1=\mu_2$ and $v_1^2=v_2^2$, we have that the first term in \eqref{E:lya30} becomes $M(u)$, whose Poisson bracket vanishes.  We included the localization in this term to illustrate the difficulty in treating the asymmetric case -- one would not have that the first term in \eqref{E:lya34} is $O(h^5)$.}  Hence
\begin{equation}
\label{E:423}
|\partial_t L_z(u) | \lesssim  ((|v_1|+|v_2|) \delta + h)  \|\tilde w\|_{H^1}^2 +h^3 \|\tilde w\|_{H^1} + \|\tilde w\|_{H^1}^3
\end{equation}
Since $|v_j|\lesssim h^{-1}\log h^{-1}$ and $\delta \sim (\log h^{-1})^{-1}$, the term $(|v_1|+|v_2|) \delta \lesssim h$.

Now we turn to the matter of obtaining a lower bound for $L_z(u)$.
First note that
$$\la W_z''(\tilde u_z)\tilde w, \tilde w\ra = L_z(u) + O(\|\tilde w\|_{H^1}^3)\,.$$
Given that $\|\tilde \nu _z \|_{H^1} \lesssim h^{2-}$, we have
\begin{equation}
\label{E:lya80}
\la W_z''(u_z)w, w\ra = L_z(u) + O(h^{4-}) + O(h^{0+})\|w\|_{H^1}^2\,.
\end{equation}
The needed lower bound for the left-hand side will be established below in Lemma \ref{L:coercivity}.

For the single-soliton case, we have coercivity for the classical
functional from Weinstein \cite{W}, which we now recall.   Taking $\eta=\eta(\cdot, \mu,a,\theta,v)$ and
\begin{align*}
R_{(\mu,a,\theta,v)}(u) 
&\defeq -\frac{\partial
H(\eta)}{\partial \mu} M(u) - \frac{\partial
H(\eta)}{\partial v} P(u) + H(u)  \\
&= \frac12(\mu^2+\mu^{-2}v^2) M(u) - \mu^{-1} v P(u) + H(u)
\end{align*}
then
\begin{equation}
\label{E:lya50}
\|w\|_{H^1}^2 \lesssim \la R''(\eta) w, w \ra ,
\end{equation}
provided we assume the orthogonality conditions
\begin{align*}
&\la w, J^{-1} \partial_\mu \eta \ra =0\,,
&&\la w, J^{-1} \partial_a \eta \ra =0 \,, \\
&\la w, J^{-1} \partial_\theta \eta \ra =0\,, 
&&\la w, J^{-1} \partial_v \eta \ra =0 \,.
\end{align*}
A direct proof of \eqref{E:lya50} is possible; see \cite[Prop. 4.1]{HZ1}.

We now prove a similar argument for the double-soliton functional $W_z(u)$ defined in \eqref{E:lya30}.  Before proceeding, we record the formulae
\begin{equation}
\label{E:lya51}
\begin{aligned}
&M_j''(u) = \psi_j \\
&P_j''(u) = -i\psi_j^{1/2}\partial_x \psi_j^{1/2} = - \tfrac12 i \psi_j' - i \psi_j \partial_x \\
& H''(u) = -\tfrac12 \partial_x^2 - 2|u|^2 - u^2 C
\end{aligned}
\end{equation}
where $C$ denotes the operator of complex conjugation.

\begin{lemma}
\label{L:coercivity} 
Suppose $w$ satisfies  the orthogonality conditions
\eqref{E:orth}.  Then
 \begin{equation}
\label{E:lya60}
\|w\|_{H^1}^2 \lesssim \la W_z''(u_z)w,w\ra \,.
\end{equation}
\end{lemma}
\begin{proof}
Denote $w_j=\psi_j^{1/2} w$,
$j=1,2$.   Note that $w_1+w_2 \neq w$, although $1=\psi_1+\psi_2 \leq \psi_1^{1/2}+\psi_2^{1/2} \leq 2$.  Define functionals
$$
W_j(u) 
\begin{aligned}[t]
&= -\frac{\partial H(\eta_j)}{\partial \mu_j} M(u) - \frac{\partial H(\eta_j)}{\partial v_j} P(u) + H(u) \\
&= \tfrac12(\mu_j^2+\mu_j^{-2}v_j^2)M(u) - \mu_j^{-1}v_j P(u) + H(u)
\end{aligned}
$$
We claim that
\begin{equation}
\label{E:est-varepsilon}
\left| \la W''(u_z)w,w\ra - \sum_{j=1}^2 \la W_j''(\eta_j)w_j,w_j\ra \right| \lesssim \delta^2 \|w\|_{H^1}^2
\end{equation}
and
\begin{equation}
\label{E:est-H1}
\left| \|w\|_{H^1}^2 - \sum_{j=1}^2 \|w_j\|_{H^1}^2 \right| \lesssim \delta^2 \|w\|_{H^1}^2
\end{equation}
We now establish \eqref{E:est-varepsilon}.
Note that (see \eqref{E:lya51})
\begin{align*}
\indentalign \la W''(u_z)w,w\ra - \sum_{j=1}^2 \la W_j''(\eta_j)w_j,w_j\ra \\
&= \la (H''(u_z)-\psi_1^{1/2}H''(\eta_1)\psi_1^{1/2} - \psi_2^{1/2}H''(\eta_2)\psi_2^{1/2}) w, w\ra
\end{align*}
The operator appearing on the right-hand side can be decomposed into $A_1+A_2+A_3$ where
\begin{align*}
&A_1 \defeq -\frac12 (\partial_x^2 - \psi_1^{1/2}\partial_x^2 \psi_1^{1/2} - \psi_2^{1/2} \partial_x^2 \psi_2^{1/2}) \\
&A_2 \defeq -2(|u_z|^2 - \psi_1|\eta_1|^2 - \psi_2|\eta_2|^2) \\
&A_3 \defeq -(u_z^2 - \psi_1 \eta_1^2 - \psi_2 \eta_2^2) C
\end{align*}
We compute $A_1$ explicitly:
$$A_1 = \sum_{j=1}^2 (-\frac12 \psi_j' \partial_x - \frac14 \psi_j^{-1}(\psi_j')^2 -\frac12 \psi_j'') = - \frac14 \sum_{j=1}^2\psi_j^{-1}(\psi_j')^2$$
where we have used that $\psi_1+\psi_2 =1$ in the second equality.  We have $(\psi_j')^2 \lesssim \delta^2 \psi_j$ by the corresponding property of $\Psi$ and thus $A_1$ is a multiplication operator with symbol bounded by $\delta^2$.   By the support properties of $\psi_1$, $\psi_2$, we obtain that the multiplication operators $A_2$, $A_3$ have symbols bounded by $h$.  This completes the proof of \eqref{E:est-varepsilon}, and the proof of \eqref{E:est-H1} is similar.

By the orthogonality conditions \eqref{E:orth},
$$\la w_1, J^{-1} \partial_{z^j} u_z \ra = - \la (1-\psi_1^{1/2})w, J^{-1} \partial_{z^j}u_z\ra\,,$$ 
and we have, for example
$$\la w_1,J^{-1} \partial_{\mu_1} u_z \ra =  - \la (1-\psi_1^{1/2})w,J^{-1} \partial_{\mu_1} u_z \ra = - \la (1-\psi_1^{1/2})w,J^{-1} \partial_{\mu_1} \eta_1 \ra\lesssim
h^{1/2} \|w\|_{L^2}\,,$$
due to the fact that $\| (1-\psi_1^{1/2}) J^{-1} \partial_{\mu_1} \eta_1 \|_{L^2} \lesssim h^{1/2}$.
Hence, by the coercivity of the classical Lyapunov
functional (see discussion surrounding \eqref{E:lya50}), we have that
$$\sum_{j=1}^2 \la W_j''(\eta_j)w_j,w_j\ra +  h \|w\|_{L^2}^2 \gtrsim \sum_{j=1}^2
\|w_j\|_{H^1}^2\,,$$  
From this and  \eqref{E:est-varepsilon}, \eqref{E:est-H1}, we obtain \eqref{E:lya60}.
\end{proof}

\section{Conclusion of proof}
\label{S:conclusion}

In this section, we conclude the proof of Theorem \ref{T:main1}.  

Recall that $h=e^{-a_0}$ which implies that $a_0 = \log h^{-1}$, and that we are in the even/odd solution setting with \eqref{E:V-11}, \eqref{E:V-12} in place.

We introduce $0<\delta \ll 1$.  The constant $\delta$ is absolute and is chosen sufficiently small in terms of the accumulation of numerous other absolute constants appearing in several estimates.  In our argument, $c$ will represent a large absolute constant that may change (typically enlarge) from one line to the next. At the conclusion of the argument, we can finally declare that $\delta$ should be taken small enough that $c\delta < \frac12$.  This does not constitute circular reasoning since one could tally up all of the absolute constants (the $c$'s) in each estimate in advance of executing the argument and suitably define $\delta$ \emph{a priori} but this is not a practical manner of exposition.

Recall that we started by defining
$$w(t) = u(t) - u_{z(t)}$$
where $z$ was selected by the implicit function theorem so that orthogonality conditions \eqref{E:orth} hold.  By continuity of the flow in $H^1$, this is possible at least up to some small positive time.  Let $T$ be the supremum of all times $0<T\leq h^{-1-\delta}$ for which 
\begin{align}
\label{E:BS1} &\| w\|_{L_{[0,T]}^\infty H_x^1} \leq h^{3/2} \\
\label{E:BS2} &|v| \leq h^{1-\delta} \\
\label{E:BS3} &a \geq a_0^{1-\delta}
\end{align}
Note that the requirement \eqref{E:BS1} implies
$$\|w\|_{H^1}^3 \leq h \|w\|_{H^1}^2 + h^3 \|w\|_{H^1} \,,$$
and enables us to discard cubic error terms in $w$ in our estimates.

In the course of the argument that follows, we work on the time interval $[0,T]$.  At the conclusion of the argument, we are able to assert that either $T=\delta h^{-1}\log h^{-1}$ or that \eqref{E:BS2} or \eqref{E:BS3} fail to hold at $t=T$.

It follows from the decomposition \eqref{E:422}
and the bootstrap assumptions \eqref{E:BS2}, \eqref{E:BS3} above that
(see Appendix \ref{A:computations}) 
$$ \sup_{1\leq n \leq 8} \| J^{-1} \partial_{z^n}\Pi_z^\perp JH'(u_z) \|_{H_x^1} \leq h^{2-c\delta}$$
Let 
$$\epsilon \defeq h^{4-c\delta} + \|w\|_{L_{[0,T]}^\infty H_x^1}^2$$
By Lemma 3.1 and the computations in Appendix \ref{A:computations}, the ODEs
$$
\left\{
\begin{aligned}
&\dot \mu = (-1)^\sigma (8a-4)ve^{-2a} + O(\epsilon)\\
&\dot a = \mu^{-1}v + (-1)^\sigma (-4a+\frac23\pi^2)ve^{-2a}+  O(\epsilon)\\
&\dot \theta = \frac12 \mu^2 + \frac12 v^2 \mu^{-2} + 18(-1)^\sigma e^{-2a} +  O(\epsilon)\\
&\dot v=-4(-1)^\sigma e^{-2a} + O(\epsilon)
\end{aligned}
\right.
$$
hold on $[0,T]$.    By the first of these equations and \eqref{E:BS1}, \eqref{E:BS2}, \eqref{E:BS3}, we have $|\mu -1 | \leq h^{2-c\delta}$.  From the above ODEs and \eqref{E:BS1}, \eqref{E:BS2}, \eqref{E:BS3}, we can deduce bounds on $\dot \mu$, $\dot a$, $\dot \theta$, and $\dot v$ that justify the estimates involved in the construction of $\nu_z$ in \S 4 summarized in Lemma 4.2.  The result is that
\begin{equation}
\label{E:424}
\| \nu_z \|_{H_x^1} \lesssim h^{2-c\delta}
\end{equation}
and \eqref{E:approx-sol-1} holds with 
$$\|  F \|_{H_x^1} \lesssim \epsilon \,.$$
By \eqref{E:423},
\begin{equation}
\label{E:425}
|\partial_t L_z(u)| \lesssim h \|\tilde w\|_{H^1}^2 + h^3 \|\tilde w\|_{H^1} + \|\tilde w \|_{H^1}^3
\end{equation}
where we recall that $\tilde w = w - \nu_z$.
By \eqref{E:424}, $\|\tilde w\|_{H^1} \lesssim \| w\|_{H^1} + h^{2-c\delta}$, we obtain from \eqref{E:425} that
\begin{equation}
\label{E:426}
|\partial_t L_z(u)| \lesssim h\|w\|_{H^1}^2 + h^{5-c\delta} \,.
\end{equation}
From \eqref{E:lya60} and \eqref{E:lya80}, the bound
\begin{equation}
\label{E:427}
\|w\|_{H^1}^2 \lesssim L_z(u) + h^{4-c\delta}
\end{equation}
holds.  Combining \eqref{E:426} and \eqref{E:427}, we obtain the bound
$$|\partial_t L_z(u)| \lesssim hL_z(u) + h^{5-c\delta}$$
By Gronwall's inequality, it follows that
$$L_z(u) \lesssim e^{cth}h^{4-c\delta} \,.$$
Provided we restrict to $t \lesssim \delta h^{-1} \log h^{-1}$, this implies
$$L_z(u) \lesssim h^{4-c\delta}$$
Reapplying \eqref{E:427}, we obtain
\begin{equation}
\label{E:430}
\|w \|_{H^1} \lesssim h^{2-c\delta} \,.
\end{equation}
At this point, we can declare that $\delta$ should have been taken sufficiently small so that $c\delta < \frac12$, where $c$ is as it appears in \eqref{E:430}.  It follows that \eqref{E:BS1}, \eqref{E:BS2}, \eqref{E:BS3} can only break down provided $T\gtrsim \delta h^{-1}\log h^{-1}$ or if either \eqref{E:BS2} or \eqref{E:BS3} fails at $t=T$.

We will see the from the following ODE analysis that \eqref{E:BS3} always holds;
in the same phase (even solution, attractive) case, the assumption \eqref{E:BS2} first fails at $T\sim h^{-1}$, and in the opposite phase (odd solution, repulsive) case, \eqref{E:BS2} remains valid and we can reach $T\sim h^{-1}\log h^{-1}$.

Since we now restrict to $t \lesssim \delta h^{-1} \log h^{-1}$, we can assume that \eqref{E:430} holds and thus $\epsilon\lesssim h^{4-c\delta}$.

Let $\tilde z = (\tilde \mu, \tilde a, \tilde \theta, \tilde v)$ solve
$$
\left\{
\begin{aligned}
&\dot {\tilde \mu} = (-1)^\sigma (8\tilde a-4)\tilde ve^{-2\tilde a} \\
&\dot {\tilde a} = \tilde v \\
&\dot {\tilde \theta} = \frac12 \tilde \mu^2 + \frac12 \tilde v^2 \tilde \mu^{-2} + 18(-1)^\sigma e^{-2\tilde a} \\
&\dot {\tilde v}=-4(-1)^\sigma e^{-2\tilde a} 
\end{aligned}
\right.
$$
These tilde equations appear in the statement of Theorem \ref{T:main1} without tildes.  Note that the $\dot {\tilde a}$ and $\dot {\tilde v}$ equations can be solved separately as discussed in \S\ref{S:introduction}.  
Let $\bar a = \mu a - \tilde a$ and $\bar v = v - \tilde v$.  Then we get the system
$$
\left\{
\begin{aligned}
&\dot{\bar a} = \bar v + O(h^{3-c\delta}) \\
&\dot{\bar v} = -2 (-1)^\sigma e^{-2\tilde a}\bar a + O(h^{4-c\delta})
\end{aligned}
\right.
$$
Let $\gamma = (\bar a)^2 + h^{-2}\bar v^2$.  Then, substituting
$$\dot \gamma \lesssim h \bar a (h^{-1}\bar v) + (h^{1/2}\bar a) h^{\frac52 - c\delta} + (h^{-1/2}\bar v) (h^{\frac52-c\delta})$$
By the inequality $\alpha\beta \leq \alpha^2 + \beta^2$, we obtain
$$\dot \gamma \lesssim h\gamma +  h^{5-c\delta}$$
By Gronwall's inequality,
$$\gamma \lesssim e^{cht}(\gamma_0 + h^{4-c\delta})$$
It follows that
$$|\bar a| \lesssim h^{2-\delta}, \qquad |\bar v| \lesssim h^{3-c\delta}$$
These errors only affect the $\dot \mu$ equation at order $h^{4-c\delta}$ so $\mu$ is only affected at order $h^{3-c\delta}$. Given this, the $\dot \theta$ equation is only affected at order $h^{3-c\delta}$.  Thus, the impact on $\theta$ is of size $h^{2-c\delta}$.  In conclusion
$$ |\bar \theta | \lesssim  h^{2-c\delta} \, \qquad |\bar \mu| \lesssim h^{3-c\delta}$$
Thus 
$$\| u_z - u_{\tilde z} \|_{H^1} \lesssim h^{2-c\delta}$$
Since $u_z$ in Theorem \ref{T:main1} in fact means $u_{\tilde z}$, this completes the proof of Theorem \ref{T:main1}.

\appendix

\section{Computations}
\label{A:computations}

We shall carry out the computations of the ODEs appearing in \eqref{E:eff-dyn} in Lemma \ref{L:ODEs} and show that they are equivalent to \eqref{E:V-3}, with errors of size $O(h^{4-})$.  This is carried out without making the symmetry assumption on the solution.  When the even/odd symmetry assumption is imposed, we will carry out the integrals appearing in \eqref{E:V-3} and show that the ODEs claimed in the statement of Theorem \ref{T:main1} hold.

Denote $u_z=\eta_1+\eta_2$. Let $L = \{1,2,5,6\}$ denote the indices that refer to the left soliton and $R=\{3,4,7,8\}$ denote the indices that refer to the right soliton.
The coefficient matrix of the symplectic form is
$$(a_{\ell m}) = A = 
\begin{bmatrix}
0 & -I \\
I & 0 
\end{bmatrix}
+ O(h^{2-})
$$ 
where the $O(h^{2-})$ contributions come from $a_{\ell m}$ with $\ell \in L$ and $m\in R$ (and vice-versa, but of course $a_{\ell m} = -a_{m \ell}$).  Fortunately, we do not need to compute these terms.  Note that
$$(a^{\ell m}) = A^{-1} = 
\begin{bmatrix}
0 & I \\
-I & 0 
\end{bmatrix}
+ O(h^{2-})
$$ 

In fact, we can substantially reduce the complexity of computation in applying Lemma \ref{L:ODEs} by observing that $JH'(u_z)$ decomposes into terms parallel to $M$ plus other terms which are $O(h^{2-})$.   To this end, we expand:
$$H'(u_z) = H'(\eta_1)+H'(\eta_2)+H_p''(\eta_1)\eta_2 + H_p''(\eta_2)\eta_1 +O(h^4)\,,$$
where
$$H_p(u) = -\frac14 \int |u|^4 \,.$$
Moreover, we have 
$$JH'(\eta) = \partial_v H(\eta) \partial_a \eta - \partial_\mu H(\eta) \partial_\theta \eta \,.$$
Hence, 
\begin{equation}
\label{E:ApA1}
H'(u_z) = \sum_{j=1}^8 b^j J^{-1}\partial_{z^j}u_z + H_p''(\eta_1)\eta_2 + H_p''(\eta_2)\eta_1
\end{equation}
where
$$
b_2 = \partial_{v_1}H(\eta_1)\,, \quad 
b_4 = \partial_{v_2}H(\eta_2) \,, \quad 
b_5 = -\partial_{\mu_1}H(\eta_1) \,, \quad 
b_7 = -\partial_{\mu_2}H(\eta_2)$$
and all other $b_j=0$.
Observe that $\la H_p''(\eta_1)\eta_2, \partial_{z^\ell}u_z \ra=O(h^{4-})$ for any $\ell\in R$ and $\la H_p''(\eta_2)\eta_1, \partial_{z^\ell}u_z \ra=O(h^{4-})$ for any $\ell \in L$.  Note further that for $\ell \in L$ (and hence $\partial_{z^\ell} u_z = \partial_{z^\ell} \eta_1$) we have
\begin{equation}
\label{E:ApA2}
\la H_p''(\eta_1)\eta_2, \partial_{z^\ell}u_z\ra = \la \eta_2, H_p''(\eta_1)\partial_{z^\ell} \eta_1\ra = \partial_{z^\ell}\la \eta_2, H_p'(\eta_1)\ra
\end{equation}
Similarly, for $\ell\in R$ and (and hence $\partial_{z^\ell} u_z = \partial_{z^\ell} \eta_2$) we have
\begin{equation}
\label{E:ApA3}
\la H_p''(\eta_2)\eta_1, \partial_{z^\ell}u_z\ra = \la \eta_1, H_p''(\eta_2)\partial_{z^\ell} \eta_1\ra = \partial_{z^\ell}\la \eta_1, H_p'(\eta_2)\ra
\end{equation}
From \eqref{E:ApA1},\eqref{E:ApA2}, and \eqref{E:ApA3}, we obtain
$$\partial_{z^\ell} H(u_z) = -\sum_{j=1}^8 b^j a_{j\ell} + \partial_{z^\ell}\la \eta_1, H_p'(\eta_2)\ra + \partial_{z^\ell} \la \eta_2, H_p'(\eta_1)\ra $$
It follows that the equations \eqref{E:eff-dyn2} reduce to
$$\dot z^m= b^m - \sum_{\ell\in L} \partial_{z^\ell} \la H_p'(\eta_1),\eta_2\ra a^{\ell m} - \sum_{\ell\in R} \partial_{z^\ell} \la H_p'(\eta_2),\eta_1\ra a^{\ell m} + O(h^{4-})$$
It suffices in this sum to discard $O(h^{2-})$ terms in $a^{\ell m}$. Thus we obtain the equations
$$
\left\{
\begin{aligned}
&\dot z^1 = b^1 + \partial_{z^5} \la H_p'(\eta_1), \eta_2\ra + O(h^{4-}) \\
&\dot z^2 = b^2 + \partial_{z^6} \la H_p'(\eta_1), \eta_2\ra + O(h^{4-}) \\
&\dot z^3 = b^3 + \partial_{z^7} \la H_p'(\eta_2), \eta_1\ra + O(h^{4-}) \\
&\dot z^4 = b^4 + \partial_{z^8} \la H_p'(\eta_2), \eta_1\ra + O(h^{4-}) \\
&\dot z^5 = b^5 - \partial_{z^1} \la H_p'(\eta_1), \eta_2\ra + O(h^{4-}) \\
&\dot z^6 = b^6 - \partial_{z^2} \la H_p'(\eta_1), \eta_2\ra + O(h^{4-}) \\
&\dot z^7 = b^7 - \partial_{z^3} \la H_p'(\eta_2), \eta_1\ra + O(h^{4-}) \\
&\dot z^8 = b^8 - \partial_{z^4} \la H_p'(\eta_2), \eta_1\ra + O(h^{4-}) 
\end{aligned}
\right.
$$
In more direct language, these equations are
\begin{alignat*}{2}
&\dot \mu_1 = && +\partial_{\theta_1} \la H_p'(\eta_1),\eta_2\ra + O(h^{4-}) \\
&\dot a_1 = +\partial_{v_1} H(\eta_1) &&+\partial_{v_1} \la H_p'(\eta_1), \eta_2\ra + O(h^{4-})\\
&\dot \mu_2 = && +\partial_{\theta_2}\la H_p'(\eta_2),\eta_1\ra + O(h^{4-})\\
&\dot a_2 =  +\partial_{v_2}H(\eta_2) &&+\partial_{v_2} \la H_p'(\eta_2), \eta_1 \ra
+ O(h^{4-}) \\
&\dot \theta_1 = -\partial_{\mu_1} H(\eta_1)&&-\partial_{\mu_1} \la H_p'(\eta_1),\eta_2\ra + O(h^{4-}) \\
&\dot v_1 =  &&-\partial_{a_1} \la H_p'(\eta_1), \eta_2\ra + O(h^{4-})\\
&\dot \theta_2 = -\partial_{\mu_2}H(\eta_2) &&-\partial_{\mu_2}\la H_p'(\eta_2),\eta_1\ra + O(h^{4-})\\
&\dot v_2 =  &&-\partial_{a_2} \la H_p'(\eta_2), \eta_1 \ra
+ O(h^{4-}) 
\end{alignat*}

We note that these equations hold in general, without assuming that the solution is even or odd.

The next step is then to compute $\la H_p'(\eta_1), \eta_2\ra$ and $\la H_p'(\eta_2),\eta_1\ra$.  Let $\phi(x)=\sech x$.  We have
\begin{equation}
\label{E:int-Ham}
\begin{aligned}
\la H_p'(\eta_2), \eta_1 \ra = &-\Re \Big( 
e^{i(\theta_2-\theta_1)} e^{i(\mu_1^{-1}v_1a_1-\mu_2^{-1}v_2a_2)} \mu_2^3 \mu_1 \\ 
& \qquad \times \int e^{i(\mu_2^{-1}v_2-\mu_1^{-1}v_1)x} \phi^3(\mu_2(x-a_2)) \phi(\mu_1(x-a_1) \, dx \Big) 
\end{aligned}
\end{equation}
At this point we will make the even/odd assumption.  In the even case, we may set
\begin{equation}
\label{E:even-defs}
(\mu, a, \theta, v) \defeq (\mu_1, -a_1, \theta_1, -v_1) = (\mu_2, a_2, \theta_2, v_2)
\end{equation}
Then $\theta_1-\theta_2 = 0$ .
In the odd case, we may set
\begin{equation}
\label{E:odd-defs}
(\mu, a, \theta, v) \defeq (\mu_1, -a_1, \theta_1-\pi , -v_1) = (\mu_2, a_2, \theta_2, v_2)
\end{equation}
Then $\theta_1-\theta_2 = \pi$. 

In either the even or odd case, we find that $\dot \mu = \dot \mu_1 =\dot \mu_2 = O(h^{3-})$, from which it follows that 
\begin{equation}
\label{E:mu}
\mu = \mu_1 = \mu_2 = 1+ O(h^{2-})
\end{equation}  Take $\sigma=0$ in the even case and $\sigma = 1$ in the odd case.  We compute the equations for $\dot \mu_2$, $\dot a_2$, $\dot \theta_2$, $\dot v_2$ by carrying out the appropriate derivative of \eqref{E:int-Ham}, and then evaluating the resulting expression using \eqref{E:mu}, \eqref{E:even-defs}, \eqref{E:odd-defs}.
By residue calculus computations and asymptotic expansion,
\begin{align*}
\alpha(\xi,a) &\defeq \int_{-\infty}^{+\infty} e^{-ix \xi} \, \phi^3 (x-a) \, \phi(x+a) \, dx \\
&= e^{-2a}[4+(2-4a)i\xi + ( -\frac{\pi^2}{6}+2a-2a^2)\xi^2+ O(\xi^3)] + O(e^{-4a})
\end{align*}
and
\begin{align*}
\beta(\xi,a) &\defeq \int_{-\infty}^{+\infty} e^{-ix \xi} \, [\phi^3]' (x-a) \, \phi(x+a) \, dx \\
&= e^{-2a}[4 + (6-4a)i\xi]+O(h^{4-})
\end{align*}

We find that
$$
\left\{
\begin{aligned}
&\dot \mu = (-1)^\sigma \Re[-i\alpha] + O(h^{4-}) \\
&\dot a = \mu^{-1}v + (-1)^\sigma\Re[+ia \alpha + \partial_\xi \alpha] + O(h^{4-})\\
&\dot \theta = \frac12\mu^2 + \frac12 v^2\mu^{-2} + (-1)^{\sigma} \Re [(iva + 3)\alpha + v(\partial_\xi\alpha)] +\Re(a\beta-i\partial_\xi\beta) +O(h^{4-})\\
&\dot v = (-1)^\sigma \Re [-iv\alpha -\beta]+O(h^{4-})
\end{aligned}
\right.
$$
where $\alpha$, $\partial_\xi\alpha$, and $\beta$ are evaluated at $\xi=-2v$.

Substituting, we obtain
$$
\left\{
\begin{aligned}
&\dot \mu = (-1)^\sigma (8a-4)ve^{-2a} + O(h^{4-})\\
&\dot a = \mu^{-1}v + (-1)^\sigma (-4a+\frac23\pi^2)ve^{-2a}+O(h^{4-})\\
&\dot \theta = \frac12 \mu^2 + \frac12 v^2 \mu^{-2} + 18(-1)^\sigma e^{-2a} +O(h^{4-})\\
&\dot v=-4(-1)^\sigma e^{-2a} + O(h^{4-})
\end{aligned}
\right.
$$

The system $(\mu a, v)$ can be solved with error $O(h^{2-})$; from which $(a,v)$ can be recovered with error $O(h^{2-})$.  At this accuracy the dynamics are comparable to 
$$
\left\{
\begin{aligned}
&\dot a = v\\
&\dot v = -4(-1)^\sigma e^{-2a}
\end{aligned}
\right.
$$
Then $\mu$ can be solved with ``explicit'' order $h^2$ term coming from the order $h^3$ term in the equation for $\dot \mu$, and then $\dot \theta$ can be obtained with error of size $h^2$.

\end{document}